\documentclass{article}

\usepackage[english]{babel} 
\usepackage[utf8]{inputenc} 
\usepackage[a4paper,left=2.5cm,right=2.5cm,bottom=3cm]{geometry} 
\usepackage{indentfirst} 
\usepackage{amsmath,amssymb,amsfonts,amsthm} 
\numberwithin{equation}{section}
\usepackage{bm} 
\usepackage{graphicx} 
\usepackage[export]{adjustbox} 
\usepackage{pdflscape} 
\usepackage{fancyhdr} 
\usepackage[colorlinks=true,linkcolor=cyan,citecolor=viridian]{hyperref} 
\usepackage{natbib} 
\setcitestyle{numbers,square}
\usepackage{flafter} 
\usepackage[framemethod=tikz]{mdframed} 
\usepackage{color} 
	\definecolor{yellow-green}{rgb}{0.6, 0.8, 0.2}
	\definecolor{viridian}{rgb}{0.25, 0.51, 0.43}
\usepackage{wrapfig} 
\usepackage{lipsum} 
\usepackage{prettyref}
\usepackage[all]{xy}
\usepackage[auth-lg,noblocks]{authblk}
\usepackage{framed}
\usepackage{tikz}
\usepackage{quiver}
\usetikzlibrary{hobby}
\usepackage{thmtools}
\usepackage{thm-restate}
\usepackage{enumitem}

\graphicspath{ {Images/} } 

\usepackage{multicol} 
\usepackage{float} 

\usepackage{nicematrix}
\usetikzlibrary{patterns}
\usetikzlibrary{matrix,decorations.pathreplacing}
\usetikzlibrary{decorations.pathreplacing,calligraphy}
\title{Compact relative $\mathrm{SO}_0(2,q)$-character varieties of punctured spheres}

\author{Yu Feng\thanks{Chern Institute of Mathematics and LPMC, Nankai University, Tianjin 300071, China, \href{mailto:yuf@nankai.edu.cn}{yuf@nankai.edu.cn}} \mbox{\ and }Junming Zhang\thanks{Chern Institute of Mathematics, Nankai University, Tianjin 300071, China, \href{mailto:junmingzhang@mail.nankai.edu.cn}{junmingzhang@mail.nankai.edu.cn}}}

\date{\vspace{-2em}}

\newtheorem{theorem}{Theorem}[section]
\newtheorem{corollary}[theorem]{Corollary}
\newtheorem{proposition}[theorem]{Proposition}
\newtheorem{lemma}[theorem]{Lemma}
\newtheorem{definition}[theorem]{Definition}
\newtheorem{remark}[theorem]{Remark}

\newrefformat{lemma}{Lemma \ref{#1}}
\newrefformat{coro}{Corollary \ref{#1}}
\newrefformat{thm}{Theorem \ref{#1}}
\newrefformat{prop}{Proposition \ref{#1}}
\newrefformat{rem}{Remark \ref{#1}}
\newrefformat{example}{Example \ref{#1}}
\newrefformat{defn}{Definition \ref{#1}}
\newrefformat{section}{Section \ref{#1}}
\newrefformat{conj}{Conjecture \ref{#1}}
\newrefformat{apdx}{Appendix \ref{#1}}
\newrefformat{fact}{Fact \ref{#1}}
\newcommand{\dd}{{\mathrm d}}
\newcommand{\EE}{{\mathcal{E}}}

\newcommand{\KK}{{\mathcal{K}}}
\newcommand{\LL}{{\mathcal{L}}}
\newcommand{\UU}{{\mathcal{U}}}
\newcommand{\VV}{{\mathcal{V}}}
\newcommand{\SO}{{\mathrm{SO}}}
\newcommand{\so}{{\mathfrak{so}}}
\newcommand{\sslash}{
	\mathchoice{\mathbin{\mkern-3mu/\mkern-6mu/\mkern-3mu}}
	{\mathbin{\mkern-3mu/\mkern-6mu/\mkern-3mu}}
	{\mathbin{\mkern-2mu/\mkern-5mu/\mkern-1mu}}
	{\mathbin{\mkern-2mu/\mkern-5mu/\mkern-1mu}}
}
\newcommand{\iu}{{\mathrm i}}
\newcommand{\dprime}{{\prime \prime}}

\begin{document}

\pagenumbering{gobble} 
\maketitle

\pagenumbering{arabic} 
\setcounter{section}{0}
\setcounter{page}{1}
\vspace{-1em}

\begin{abstract}
    We prove that there are relative $\SO_0(2,q)$-character varieties of the punctured sphere which are compact, totally non-hyperbolic and contain a dense representation. This work fills a remaining case of the results of N. Tholozan and J. Toulisse. Our approach relies on the non-abelian Hodge correspondence and we study the moduli space of parabolic $\SO_0(2,q)$-Higgs bundles with some fixed weight. Additionally, we provide a construction based on Geometric Invariant Theory (GIT) to demonstrate that the considered moduli spaces can be viewed as a projective variety over $\mathbb{C}$.
\end{abstract}

\small\textbf{Keywords. }{Character varieties, Parabolic Higgs bundles, Geometric invariant theory.}

\small\textbf{2020 Mathematics Subject Classification. }{14L24, 14M35, 14H60}

\large

\tableofcontents

\section{Introduction}

For an oriented surface $\Sigma_{g,s}$ of genus $g$ and with $s$ punctures, its $G$-character variety, where $G$ is a real reductive Lie group, consists of the equivalence classes of reductive representations from the fundamental group $\Gamma_{g,s}:=\pi_1(\Sigma_{g,s})$ to $G$. It is a classical problem to study the topology of the character variety.

For closed surfaces $\Sigma_{g,0}$ where $g>1$, and $G=\mathrm{PSL}(2,\mathbb{R})$ or $\mathrm{PSL}(2,\mathbb{C})$, the topology of the character variety is well studied by N. Hitchin in \cite{hitchin1987self} with an analytic method using the technique of Higgs bundles and by W. M. Goldman in \cite{goldman1988topological} in a more geometric way. An important result is that when $G=\mathrm{PSL}(2,\mathbb{R})$, the connected components of the character variety are distinguished by the Euler number of their representations, which is bounded by the Milnor--Wood inequality \cite{milnor1958existence}. 

However, when the surface is non-compact, as shown in \cite{deroin2019supra} by B. Deroin and N. Tholozan for $\Sigma_{0,s}$ with $s\geqslant3$, there exist \textit{supra-maximal} representations whose Euler number exceeds the bound of the Milnor--Wood inequality. Furthermore, they proved that the supra-maximal representations with prescribed monodromy form a connected component symplectomorphic to $\mathbb{C}P^{s-3}$ by using the Atiyah--Bott--Goldman symplectic structure, which generalizes Benedetto--Goldman's work \cite{benedetto1999topology} for $s=4$. Soon after their results, in \cite{mondello_2018}, G. Mondello used parabolic Higgs bundles and non-abelian Hodge theory to give a complete description of relative $\mathrm{PSL}(2,\mathbb{R})$-character
varieties of $\Sigma_{g,s}$ and reproved Deroin--Tholozan's results.

Later, by using the non-abelian Hodge correspondence, N. Tholozan and J. Toulisse generalized Deroin--Tholozan's results to $G=\mathrm{SU}(p,q)$ in \cite{tholozan2021compact}. They studied some compact components of the moduli space of parabolic $\mathrm{SU}(p,q)$-Higgs bundles and showed that some relative character varieties $\mathfrak{X}_h(\Sigma_{0,s},\mathrm{SU}(p,q))$ (the point in character variety with prescribed monodromy $h$) have a compact connected component which consists of totally non-hyperbolic representations. Moreover, their component has a representation whose image is Zariski-dense in $\mathrm{SU}(p,q)$. By embedding into $\mathrm{SU}(p,q)$, they found components with similar properties for another two families of classical Hermitian Lie groups, i.e. $\mathrm{Sp}(2n,\mathbb{R})$ and $\mathrm{SO}^*(2n)$. We refer to \cite{goldmancompact} for more information and history about compact components of planar surface group representations.

However, there is still a family of classical Hermitian Lie groups, $\SO(2,q)$, where $q\geqslant 3$, that are not covered by Tholozan--Toulisse's results. They indicated one may look directly at parabolic $\SO(2,q)$-Higgs bundles and carry out the same analysis to get similar results. 

In this article, we consider $\SO_0(2,q)$, the identity component of $\SO(2,q)$, as our group since its corresponding parabolic Higgs bundle will be easier to analyze. We will prove that there are some compact relative components in $\SO_0(2,q)$-character variety of $\Sigma_{0,s}$ which have similar properties as Deroin--Tholozan's components consisting of supra-maximal representations:

\begin{theorem}\label{thm:main}
    For any $s\geqslant3$, there exists a tuple $h=(h_1,\dots,h_s)\in T^s$, where $T$ is a fixed maximal torus of $\SO_0(2,q)$ such that the relative character variety $\mathfrak{X}_h(\Sigma_{0,s},\SO_0(2,q))$ is compact and satisfying the following properties: \begin{itemize}
        \item[(1)] It consists of totally non-hyperbolic representations, i.e. for any $[\rho]\in \mathfrak{X}_h(\Sigma_{0,s},\SO_0(2,q))$ and the homotopy class $[c]$ of any simple closed curve $c$ on $\Sigma_{0,s}$, all eigenvalues of $\rho([c])$ have modulus $1$;

        \item[(2)] It contains a dense representation, i.e. its image is dense under the Euclidean topology of $\SO_0(2,q)$;

        \item[(3)] For any $[\rho]\in \mathfrak{X}_h(\Sigma_{0,s},\SO_0(2,q))$ and every identification of $\Sigma_{0,s}$ with an $s$-punctured Riemann sphere, there is a holomorphic $\rho$-equivariant harmonic map from the universal cover $\widetilde{\Sigma_{0,s}}$ of $\Sigma_{0,s}$ to the symmetric space $(\SO(2)\times\SO(q))\backslash\SO_0(2,q)$. 
    \end{itemize} 
    In particular, $\mathfrak{X}_h(\Sigma_{0,s},\SO_0(2,q))$ has a compact connected component with the above properties. Moreover, there exists an open neighborhood $W$ of the identity element in $T^{s-1}$ such that there is a full measure subset $W^\prime\subset W$, satisfying that for any $(h_1,\dots,h_{s-1})\in W^\prime$, there exists $h_s\in T$ such that the relative character variety $\mathfrak{X}_h(\Sigma_{0,s},\SO_0(2,q))$ for $h=(h_1,\dots,h_s)$ has a compact connected component which satisfies the properties above.
\end{theorem}

Similarly to \cite{mondello_2018} and \cite{tholozan2021compact}, we first prove the above theorem in the language of Higgs bundles and then translate it into the language of representations through the non-abelian Hodge correspondence. Hence we also state our theorem in the language of Higgs bundles below. Geometric Invariant Theory is also relevant to this result. In \cite{tholozan2021compact}, N. Tholozan and J. Toulisse proved that the compact relative component they found is isomorphic to a feathered Kronecker variety which is a GIT quotient of a $\mathrm{GL}(p,\mathbb{C})\times\mathrm{GL}(q,\mathbb{C})$-action by using flag configuration. Similarly, the compact relative component we obtain in this article is also a projective variety which is constructed as a GIT quotient, and its construction is like an isotropic analogue of the feathered Kronecker variety. We will soon see that a choice of the weight of a parabolic $\SO_0(2,q)$-Higgs bundle is equivalent to a choice of an $\SO_0(2,q)$-weight (\prettyref{defn:so02qweight}) $(\alpha,\beta)$ in \prettyref{sec:so02q}. Under this setting, we prove that:

\begin{restatable}{theorem}{main}\label{thm:main2}
    For any $\SO_0(2,q)$-weight $(\alpha,\beta)$ satisfying \begin{itemize}[leftmargin=2em, itemindent=1em]
    \item[(W1)] $\alpha^j>\beta_1^j$ for all $1\leqslant j\leqslant s$;

    \item[(W2)] $|\alpha|>|\beta|$;

    \item[(W3)] $|\alpha|+|\beta|<1$,
\end{itemize} the moduli space $\mathcal{M}(\alpha,\beta)$ of polystable parabolic $\SO_0(2,q)$-Higgs bundles with weight $(\alpha,\beta)$ over the complex projective line $\mathbb{C}P^1$ with $s$ punctures is compact. Furthermore, it is a projective variety which can be realized as a GIT quotient of an $\SO(2,\mathbb{C})\times\SO(q,\mathbb{C})$-action with suitable linearization. Moreover, when $s\geqslant q+2$, it contains a stable point.
\end{restatable}

We should emphasize that the key ideas behind the proofs in \prettyref{sec:cc} and \prettyref{sec:ccrep} --- particularly the choice of suitable parabolic weights to enforce the nilpotency of the Higgs field --- are inspired by \cite{tholozan2021compact}. However, due to the orthogonal structure in our setting, several technical complications arise. Unlike in \cite{tholozan2021compact}, where the nilpotent Higgs field has only one non-vanishing block, ours retains two, significantly increasing the complexity of the stability condition.

Moreover, the stability conditions for parabolic $\SO_0(2,q)$-Higgs bundles and parabolic $\SO(2+q,\mathbb{C})$-Higgs bundles do not coincide. Consequently, the GIT construction must also be adjusted by considering isotropic flag configurations rather than ordinary ones. For further details, we refer to \prettyref{rem:diferrentstability}, \prettyref{rem:differentGIT} and \prettyref{rem:differentGITbasis}.

\paragraph{Structure of the article} We first recall some basic Lie theory of $\SO_0(2,q)$, the definition of character variety and parabolic $\mathrm{SL}(n,\mathbb{C})$-Higgs bundles in \prettyref{sec:basic}. Then we translate parabolic $\SO_0(2,q)$-Higgs bundles and their stability into the language of vector bundles in \prettyref{sec:transtovb}. In \prettyref{sec:cc}, we will prove \prettyref{thm:main2} and give the GIT construction explicitly. Finally, we conclude \prettyref{thm:main} for $s\geqslant q+2$ first and then prove it for general case by restricting to a subsurface in \prettyref{sec:ccrep}. 

\paragraph{Acknowledgements}
We are grateful to Qiongling Li for suggesting this problem and many helpful discussions. We warmly thank Oscar Garc\'ia-Prada for his explanation on his joint work \cite{biquard2020parabolic}. We also thank Hao Sun for lots of interesting discussions on the Betti moduli space and the non-abelian Hodge correspondence. We wish to thank the referee for their valuable comments, which helped to improve the exposition substantially. Both authors are partially supported by the National Key R\&D Program of China No. 2022YFA1006600, the Fundamental Research Funds for the Central Universities and Nankai Zhide Foundation.

\section{Preliminaries}\label{sec:basic}

\subsection{Character varieties}

In this subsection, we fix a real reductive Lie group $G$ with its Lie algebra $\mathfrak{g}:=\operatorname{Lie}(G)$ and maximal compact subgroup $H$, Cartan involution $\theta$ and invariant bilinear form $B=\langle\cdot,\cdot\rangle$ on $\mathfrak{g}$. Given an element $g\in G$, we denote $C(g)$ by the conjugacy class of $g$ in $G$. Let $\Sigma_{g,s}$ denote the oriented surface of genus $g$ with $s$ punctures. We denote by $\Gamma_{g,s}:=\pi_1(\Sigma_{g,s})$ its fundamental group. 

\begin{definition}
    A homomorphism $\rho\colon\Gamma_{g,s}\to G$ is called a \textbf{reductive representation} if $\operatorname{Ad}\circ\rho\colon \Gamma_{g,s}\to\mathrm{GL}(\mathfrak{g})$ decomposes as a direct sum of irreducible representations, that is, it is completely reducible.
\end{definition}

Let $\mathrm{Hom}^+(\Gamma_{g,s},G)$ the set of all reductive representations of $\Gamma_{g,s}$ in $G$. There is a moduli space of reductive representations of $\Gamma_{g,s}$ in $G$, called \textbf{(absolute) character variety}, defined as
\[\mathfrak{X}(\Sigma_{g,s},G):=\mathrm{Hom}^+(\Gamma_{g,s},G)/G,\]
where $G$ acts on $\operatorname{Hom}^+(\Gamma_{g,s},G)$ by conjugation. $\operatorname{Hom}^+(\Gamma_{g,s},G)$ and $\mathfrak{X}(\Sigma_{g,s},G)$ are equipped with the topology induced from $G$.

Let $c_1,\dots,c_s$ denote homotopy classes of loops going counter-clockwise around the punctures of $\Sigma_{g,s}$. And we fix an $h=(h_1,\dots,h_s)\in G^s$.

\begin{definition}
    A homomorphism $\rho\colon\Gamma_{g,s}\to G$ is of type $h$ if $\rho(c_j)\in C(h_j)$ for all $1\leqslant j\leqslant s$. The \textbf{relative character variety of type $h$} is defined as
$$\mathfrak{X}_h(\Sigma_{g,s},G):=\{\rho\in\mathrm{Hom}^+(\Gamma_{g,s},G)\mid \rho(c_j)\in C(h_j),j=1,\dots,s\}/G.$$
\end{definition}

\subsection{Basic Lie theory of \texorpdfstring{$\mathrm{SO}_0(2,q)$}{SO0(2,q)}}\label{sec:so02q}
We first use the standard non-degenerate bilinear form
$$\begin{aligned}
Q\colon \mathbb{R}^{2+q}\times\mathbb{R}^{2+q}&\longrightarrow\mathbb{R}\\
\left(\begin{pmatrix}
x_1\\ \vdots\\x_{2+q}\end{pmatrix},\begin{pmatrix}
y_1\\ \vdots\\y_{2+q}\end{pmatrix}
\right)&\longmapsto -x_1y_1-x_2y_2+\sum_{i=1}^qx_{2+i}y_{2+i}
\end{aligned}$$
of signature $(2,q)$ to get the group
$$\begin{aligned}\mathrm{SO}(2,q)&=\left\{A\in \mathrm{SL}(2+q,\mathbb{R})\mid Q(x,y)=Q(Ax,Ay),\forall x,y\in\mathbb{R}^{2+q}\right\}\\&=\left\{A\in\mathrm{SL}(2+q,\mathbb{R})\mid A^{\mathrm{t}}I_{2,q}A=I_{2,q}\right\},\end{aligned}$$
where $$I_{2,q}=\begin{pmatrix}
    -I_2& 0\\0& I_q
\end{pmatrix}.$$
Then $\mathrm{SO}_0(2,q)$ is defined as the identity component of $\mathrm{SO}(2,q)$. Its Lie algebra is
$$\begin{aligned}
    \so(2,q):=&\operatorname{Lie}(\mathrm{SO}_0(2,q))=\{A\in\mathfrak{sl}(2+q,\mathbb{R})\mid A^{\mathrm{t}}I_{2,q}+I_{2,q}A=0\}\\
    =&\left\{\begin{pmatrix}
        A_{11}& A_{12}\\A_{21}& A_{22}
    \end{pmatrix}\in\mathfrak{sl}(2+q,\mathbb{R})\Bigg| A_{11}+A_{11}^{\mathrm{t}}=0,A_{22}+A_{22}^{\mathrm{t}}=0,A_{21}=A_{12}^{\mathrm{t}}\right\}
\end{aligned}$$
Below we denote that $G=\SO_0(2,q)$, $\mathfrak{g}=\so(2,q)$. We fix $H=\SO(2)\times\SO(q)$ as the maximal compact subgroup of $G$, and $\mathfrak{h}:=\operatorname{Lie}(H)=\so(2)\oplus\so(q)$. To get the complexified Cartan decomposition, we use the isomorphism between $\mathfrak{g}^{\mathbb{C}}$ and $\so(2+q,\mathbb{C})$. Explicitly, $\so(2,q)$ can be viewed as a subalgebra of $\so(2+q,\mathbb{C})$ via the map
$$\begin{aligned}
    \begin{pmatrix}
        A_{11}& A_{12}\\A_{21}& A_{22}
    \end{pmatrix}\longmapsto\begin{pmatrix}
        A_{11}& \iu A_{12}\\-\iu A_{21}& A_{22}
    \end{pmatrix},
\end{aligned}$$
which means changing $Q$ in to the standard inner product on $\mathbb{C}^{2+q}$ by multiplying $\iu$ on the first two coordinates. Thus the complexified Cartan decomposition of $\mathfrak{g}$ can be expressed as
\[\begin{tikzcd}
	{\mathfrak{g}^{\mathbb{C}}\cong\so(2+q,\mathbb{C})} & {\mathfrak{h}^{\mathbb{C}}} & {\mathfrak{m}^{\mathbb{C}}} \\
	{\begin{pmatrix}         A_{11}& \iu A_{12}\\-\iu A_{21}& A_{22}     \end{pmatrix}} & {\begin{pmatrix}         A_{11}& 0\\0& A_{22}     \end{pmatrix}} & {\begin{pmatrix}         0&\iu A_{12}\\-\iu A_{21}& 0     \end{pmatrix}}
	\arrow["\in"{marking}, draw=none, from=2-1, to=1-1]
	\arrow["\oplus"{marking, pos=0.5}, draw=none, from=1-2, to=1-3]
	\arrow["\in"{marking}, draw=none, from=2-2, to=1-2]
	\arrow["{+}"{marking, pos=0.6}, draw=none, from=2-2, to=2-3]
	\arrow["\in"{marking}, draw=none, from=2-3, to=1-3]
	\arrow["{=}"{marking, pos=0.25}, draw=none, from=1-1, to=1-2]
	\arrow["{=}"{marking, pos=0.5}, draw=none, from=2-1, to=2-2]
\end{tikzcd}\]
where $A_{11}$ and $A_{22}$ are skew-symmetric complex matrices and $A_{12}=A_{21}^{\mathrm{t}}$ is a complex $(2\times q)$-matrix. Therefore we can view $\mathfrak{m}^{\mathbb{C}}$ as
$$\left\{\begin{pmatrix}
    0&B\\-B^{\mathrm{t}}&0
\end{pmatrix}\Bigg|B\in\mathbb{C}^{2\times q}\right\}.$$

Now we fix a Cartan subalgebra \begin{equation}\label{eq:Cartan}
    \mathfrak{t}=\begin{cases}
\left\{\left.T(2\pi\alpha,2\pi\beta_i)=2\pi\cdot\begin{pmatrix}
    M_{\alpha}& & &\\
    & M_{\beta_1}& &\\
    & &\ddots&\\
    & & & M_{\beta_n}
\end{pmatrix}\right|\alpha,\beta_i\in\mathbb{R}\right\}& \mbox{when }q=2n\\
&\\
\left\{\left.T(2\pi\alpha,2\pi\beta_i)=2\pi\cdot\begin{pmatrix}
    M_{\alpha}& & & &\\
    & M_{\beta_1}& & &\\
    & &\ddots& &\\
    & & & M_{\beta_n}&\\
    & & & & 0
\end{pmatrix}\right|\alpha,\beta_i\in\mathbb{R}\right\}& \mbox{when }q=2n+1
\end{cases}
\end{equation}
where $$M_\lambda=\begin{pmatrix}
    0&-\lambda\\ \lambda& 0
\end{pmatrix},\forall\lambda\in\mathbb{C}.$$
Then the corresponding root system of $(\mathfrak{h},\mathfrak{t})$ is
$$\Delta=\begin{cases}
    \{\pm e_i\pm e_j\mid 1\leqslant i<j\leqslant n\}& \mbox{when }q=2n\\
    \{\pm e_i\pm e_j\mid 1\leqslant i<j\leqslant n\}\cup\{\pm  e_k\mid 1\leqslant k\leqslant n\}& \mbox{when }q=2n+1
\end{cases}\subset\mathfrak{t}^\vee,$$
where $e_i$ maps $T(2\pi\alpha,2\pi\beta_i)$ to $\beta_i$. Now we fix the corresponding system of positive real roots is
$$\Delta^+=\begin{cases}
    \{e_i\pm e_j\mid 1\leqslant i<j\leqslant n\}& \mbox{when }q=2n\\
    \{e_i\pm e_j\mid 1\leqslant i<j\leqslant n\}\cup\{e_k\mid 1\leqslant k\leqslant n\}& \mbox{when }q=2n+1
\end{cases}\subset\mathfrak{t}^\vee.$$

\subsection{Parabolic Higgs bundles}\label{sec:GHiggs}

In \cite{biquard2020parabolic}, O. Biquard, O. Garc\'{i}a-Prada and I. Mundet i Riera introduced the concept of parabolic $G$-Higgs bundles over a Riemann surface $X$ with marked points $D$ for an arbitrary real reductive group $(G,H,\theta,B=\langle\cdot,\cdot\rangle)$. Let $X$ be a closed Riemann surface with finite marked points $\{x_i\}_{i=1}^s=:D$ on it (we will also use $D$ to denote the divisor $\sum_{i=1}^sx_i$ on $X$). Denote by $\KK$ the canonical line bundle of $X$. Any vector bundle or principal bundle we mention below is holomorphic. We mainly focus on the case $G=\mathrm{SL}(n,\mathbb{C})$ in this section. 

\begin{definition}
            Suppose $V$ is a $\mathbb{C}$-linear space. A sequence of subspaces of $V$
            \[0=F_{k}\subsetneq F_{k-1}\subsetneq\cdots\subsetneq F_2 \subsetneq F_{1}=V, \quad(\mbox{resp. }0=F_{1}\subsetneq F_{2}\subsetneq\cdots \subsetneq F_{k-1}\subsetneq F_{k}=V)\]
            is called a \textbf{reverse flag} (resp. \textbf{flag}). If $V$ is equipped with a bilinear form $Q$, then the above reverse flag (resp. flag) is called a \textbf{reverse isotropic flag} (resp. \textbf{isotropic flag}) if every $F_i$ is isotropic or coisotropic under $Q$ and $F_i=(F_{k+1-i})^{\perp_Q}$.
            
            We say a reverse flag $\left(F_i\right)_{i=1}^k$ is \textbf{monotonically weighted} by a tuple of real numbers $\left(\alpha_j\right)_{j=1}^{\dim(V)}$ if $\alpha_j\geqslant\alpha_{j+1}$, 
            \[\alpha_{\dim(V)-\dim(F_i)+1}=\alpha_{\dim(V)-\dim(F_i)+2}=\cdots=\alpha_{\dim(V)-\dim(F_{i+1})}=:\widetilde{\alpha}_i\]
            and $\widetilde{\alpha}_i>\widetilde{\alpha}_{i+1}$
            for any $i=1,\dots,k-1$. 
            
            A basis $\{e_1,\dots,e_{\dim(V)}\}$ of $V$ is called \textbf{compatible with a reverse flag $\left(F_i\right)_{i=1}^{k}$} if \[e_{\dim(V)-\dim(F_{i})+1},\dots,e_{\dim(V)}\] span $F_i$ for any $i=1,\dots,k-1$.
        \end{definition}

        For our convenience, we usually set $\widetilde{\alpha}_0=\alpha_0=\alpha_{\dim(V)+1}=0$.

For a parabolic $\mathrm{SL}(n,\mathbb{C})$-Higgs bundle $(\mathbb{E},\Phi)$, from the viewpoint of the vector bundle $\mathbb{E}[\mathbb{C}^n]$ induced by the standard representation, we obtain that a parabolic $\mathrm{SL}(n,\mathbb{C})$-Higgs bundle is equivalent to the following data (see \cite[Example 2.5 \& Section 4.1]{biquard2020parabolic}):
\begin{itemize}
    \item[(1)] a holomorphic vector bundle $\EE\to X$, with $\operatorname{rank}(\EE)=n$ and $\det(\EE)=\mathcal{O}$ which is the trivial line bundle;

    \item[(2)] a reverse flag $\EE^j=\left(\EE_i^j\right)$ of $\EE_{x_j}$ monotonically weighted by $\alpha^j=\left(\alpha_i^j\right)_{1\leqslant i\leqslant n}$ satisfying that $\alpha_i^j\in(-1/2,1/2)$ and 
    \[\sum_{i=1}^n\alpha_i^j=0\]
    for every marked point $x_j\in D$. Denote ${\alpha}:=(\alpha^j)_{j=1}^s$; 
    
    \item[(3)]
    a Higgs field $\Phi$ which is a traceless holomorphic section of $\operatorname{End}(\mathcal{E})\otimes\mathcal{K}(D)$ satisfying that
    \begin{equation}\label{eq:Higgsfield}
        \Phi=\sum_{k=1}^p\sum_{l=1}^qO\left(z^{\left\lceil\alpha_k^j-\alpha_l^j\right\rceil}\right)\cdot e_k^\vee\otimes e_l\otimes\dfrac{\dd z}{z},
    \end{equation}
     on a coordinate chart $(U,z)$ centered at $x_j$ for every $x_j\in D$, where $\{e_1,\dots,e_n\}$ is a holomorphic frame  compatible with the reverse flag $\left(\mathcal{E}^j\right)$.
\end{itemize}

Now an automorphism of a parabolic $\mathrm{SL}(n,\mathbb{C})$-Higgs bundle $(\EE,\EE^j,\alpha,\Phi)$ is an automorphism of $\EE$ which stabilizes $\Phi$ and preserves the reverse flags $\EE^j$.

The parabolic degree defined below is used to test the stability condition.

\begin{definition}
            For any holomorphic subbundle $\mathcal{E}^\prime$ of a parabolic $\mathrm{SL}(n,\mathbb{C})$-Higgs bundle $(\mathcal{E},\mathcal{E}^j,\alpha,\Phi)$, we define the \textbf{parabolic degree} of $\mathcal{E}^\prime$ as

    \[\operatorname{pardeg}(\mathcal{E}'):=\deg(\mathcal{E}')-\sum_{j=1}^{\deg(D)}\sum_{i}(\widetilde{\alpha}_{i}^j-\widetilde{\alpha}_{i-1}^j)\dim\left(\left(\mathcal{E}'\right)_{x_j}\cap\mathcal{E}_i^j\right),\]
    where we assume $\widetilde{\alpha}_{0}^j=0$.
\end{definition}

\begin{remark}
    Note that here we use ``$-$'' connecting the degree part and the parabolic part instead of ``$+$'' since we use the reverse flag and decreasing weights. 
\end{remark}

\begin{definition}
    A parabolic $\mathrm{SL}(n,\mathbb{C})$-Higgs bundle $(\mathcal{E},\mathcal{E}^j,\alpha,\Phi)$ is \textbf{stable} if for any non-trivial holomorphic subbundle $\EE^\prime\subset\EE$ which is $\Phi$-invariant, $\operatorname{pardeg}(\EE')<0$.
\end{definition}

\section{Parabolic \texorpdfstring{$\mathrm{SO}_0(2,q)$}{SO0(2,q)}-Higgs bundles via vector bundles}\label{sec:transtovb}

In this section, we apply the theory for general $G$-Higgs bundles (see \cite{biquard2020parabolic}) to $G=\mathrm{SO}_0(2,q)$, where $q$ is an integer larger than $2$, and give a translation of parabolic $\SO_0(2,q)$-Higgs bundle and its stability via vector bundle in \prettyref{prop:vbofSO02q} and \prettyref{prop:stability}. 

Let $G$ be a real reductive Lie group with maximal compact subgroup $H$, Cartan involution $\theta$ and invariant bilinear form $B=\langle\cdot,\cdot\rangle$ on $\mathfrak{g}=\operatorname{Lie}(G)$. Given the Cartan decomposition $\mathfrak{g}=\mathfrak{h}\oplus\mathfrak{m}$. A parabolic $G$-Higgs bundle over $(X,D)$ consists of a parabolic principal $H^\mathbb{C}$-bundle $\mathbb{E}$ with parabolic structures $(Q_j,\alpha^j)$ over $(X,D)$ (see \cite[Section 2.1]{biquard2020parabolic} for its definition) and a parabolic $G$-Higgs field $\Phi\in\mathrm{H}^0(P\mathbb{E}(\mathfrak{m}^{\mathbb{C}})\otimes\KK(D))$, where $P\mathbb{E}(\mathfrak{m}^{\mathbb{C}})$ denotes the sheaf of parabolic sections of $\mathbb{E}(\mathfrak{m}^\mathbb{C})$ (see \cite[Section 4.1]{biquard2020parabolic} for its definition).

\subsection{Parabolic subgroups of \texorpdfstring{$\SO(2,\mathbb{C})\times\SO(q,\mathbb{C})$}{SO(2,C)×SO(q,C)}}

Suppose the ordinary basis of $\mathbb{C}^{2+q}=\mathbb{C}^2\oplus\mathbb{C}^q$ is 
$$\mathcal{B}=\begin{cases}
    \{v,v',v_1,v_1',\dots,v_n,v_n'\}& \mbox{when }q=2n\\
    \{v,v',v_1,v_1',\dots,v_n,v_n',v_{n+1}\}& \mbox{when }q=2n+1\\
\end{cases},$$
we change it into a new basis
$$\mathcal{B}'=\begin{cases}
    \{v+\iu v',v-\iu v',v_1+\iu v_1',\dots,v_n+\iu v_n',v_n-\iu v_n',\dots,v_1-\iu v_1'\}& \mbox{when }q=2n,\\
    \{v+\iu v',v-\iu v',v_1+\iu v_1',\dots,v_n+\iu v_n',\sqrt{2}v_{n+1},v_n-\iu v_n',\dots,v_1-\iu v_1'\}& \mbox{when }q=2n+1.
\end{cases}$$
Note that every element, except $\sqrt{2}v_{n+1}$, in $\mathcal{B}'$ is isotropic. Under this basis, an element $\tau=T(2\pi\alpha,2\pi\beta_i)\in\mathfrak{t}$ defined in \prettyref{sec:so02q} can be diagonalized as
$$\begin{cases}
    \operatorname{diag}(\alpha,-\alpha,\beta_1,\dots,\beta_n,-\beta_n,\dots,-\beta_1)& \mbox{when }q=2n\\
    \operatorname{diag}(\alpha,-\alpha,\beta_1,\dots,\beta_n,0,-\beta_n,\dots,-\beta_1)& \mbox{when }q=2n+1
\end{cases}$$
We set $\beta_{q+1-i}=-\beta_{i}$ for $1\leqslant i\leqslant \lceil q/2\rceil$, then we can always say $\tau$ is $\operatorname{diag}(\alpha,-\alpha,\beta_1,\dots,\beta_q)$ under $\mathcal{B}'$.

Now we consider the parabolic subgroup $P_{\tau}=\{g\in H^\mathbb{C}\mid\operatorname{Ad}(\exp(t\tau))(g)\mbox{ is bounded}\}$ of $H^{\mathbb{C}}=\SO(2,\mathbb{C})\times\SO(q,\mathbb{C})$ defined by $\tau\in\iu\mathfrak{h}\setminus\{0\}$. Since $\tau$ can be expressed as $\operatorname{diag}(\alpha,-\alpha,\beta_1,\dots,\beta_q)$, let the matrix form of $g\in H^{\mathbb{C}}$ be $\operatorname{diag}\left(\mu,\mu^{-1},(g_{k,l})_{1\leqslant k,l\leqslant q}\right)$
under $\mathcal{B}'$, $\operatorname{Ad}(\exp(t\tau))(g)$ is $$\operatorname{diag}\left(\mu,\mu^{-1},(\mathrm{e}^{t(\beta_k-\beta_l)}g_{k,l})_{1\leqslant k,l\leqslant q}\right),$$
thus $g\in P_{\tau}$ if and only if $g_{k,l}=0$ for any $\beta_k>\beta_l$.

If we assume that
$$\beta_1=\cdots\beta_{k_1}>\beta_{k_1+1}=\cdots=\beta_{k_2}>\beta_{k_2+1}=\cdots>0\geqslant\beta_{k_t+1},$$
then we can get for any $g=\operatorname{diag}\left(\mu,\mu^{-1},(g_{k,l})_{1\leqslant k,l\leqslant q}\right)\in P_{\tau}$, $(g_{k,l})_{1\leqslant k,l\leqslant q}$ is of the form
$$\begin{pNiceMatrix}%
[margin,
name=mymatrix,
first-row,
first-col,
nullify-dots,
xdots/line-style=loosely dotted,
code-after = {}
]
&        1         & \Cdots    & k_{1}    & k_{1}+1      &\Cdots     & k_{2}   &k_{2}+1      &\Cdots    &\Cdots  & k_{t}  &k_{t}+1   &\Cdots   & q-k_{1}+1 &\Cdots  &q \\
1        & *       & \Cdots    & *        & 0            &\Cdots     & 0       &0            &\Cdots    &\Cdots  &0       &0         &\Cdots   & 0         &\Cdots  &0 \\
\Vdots   &\Vdots   & \Ddots    & \Vdots   & \Vdots       &           &\Vdots   &\Vdots       &          &        &\Vdots  &\Vdots    &         &\Vdots     &        &\Vdots  \\
k_{1}    & *       & \Cdots    & *        & 0            &\Cdots     & 0       &0            &\Cdots    &\Cdots  &0       & 0        &\Cdots   &0          &\Cdots  &0 \\
k_{1}+1  & *       & \Cdots    & *        & *            &\Cdots     & *       &0            & \Cdots   &\Cdots  & 0      &0         &\Cdots   &0          &\Cdots  &0 \\
\Vdots   & \Vdots  &           & \Vdots   & \Vdots       &           & \Vdots  &\Vdots       &          &        &\Vdots  &\Vdots    &         &\Vdots     &        &\Vdots \\
k_{2}    & *       & \Cdots    & *        & *            & \Cdots    & *       & 0           &\Cdots    &\Cdots  & 0      &0         &\Cdots   &0          &\Cdots  &0 \\
\Vdots   & \Vdots  &           & \Vdots   & \Vdots       &           & \Vdots  &\Vdots       &          &        &\Vdots  &\Vdots    &         &\Vdots     &        &\Vdots \\
&        &         &           &          &              &           &         &             &         &         &        &          &         &           &        &   \\
k_{t-1}+1   & *    & \Cdots    & *        & *            & \Cdots    & *       & *           &\Cdots   &\Cdots   & *      &0         &\Cdots   &0          &\Cdots  &0 \\  
\Vdots   & \Vdots   &          & \Vdots   & \Vdots       &           & \Vdots  &\Vdots       &          &        &\Vdots  &\Vdots    &         &\Vdots     &        &\Vdots \\
k_{t}    & *       & \Cdots    & *        & *            & \Cdots    & *       & *           &\Cdots    &\Cdots  & *      &0         &\Cdots   &0          &\Cdots  &0 \\  
\Vdots   & \Vdots  &           & \Vdots   & \Vdots       &           & \Vdots  &\Vdots       &          &        &\Vdots  &\Vdots    &         &\Vdots     &        &\Vdots \\
&        &         &           &          &              &           &         &             &         &         &        &          &         &           &        &        \\
q-k_{1}+1   & *    & \Cdots    & *        & *            & \Cdots    & *       & *           &\Cdots   &\Cdots   & *      &*         &\Cdots   &*          &\Cdots  &*       \\  
\Vdots   & \Vdots  &           & \Vdots   & \Vdots       &           & \Vdots  &\Vdots       &          &        &\Vdots  &\Vdots    &         &\Vdots     &        &\Vdots \\
q        & *       & \Cdots    & *        & *            & \Cdots    & *       & *           &\Cdots    &\Cdots  & *      &*        &\Cdots    &*          &\Cdots  &*

\end{pNiceMatrix}$$
under the basis $\mathcal{B}'$, i.e. it is the stabilizer of a reverse isotropic flag of $\mathbb{C}^q$.

\subsection{Parabolic \texorpdfstring{$\SO_0(2,q)$}{SO0(2,q)}-Higgs bundles}

Recall that $(X,D)$ is a marked Riemann Surface with $\#D=s$. Now we apply the definition of parabolic weight (c.f \cite[Definition 2.1]{biquard2020parabolic}) to $G=\SO_0(2,q)$, where $q\geqslant 2$. Let $n=\lfloor q/2\rfloor$. From \prettyref{sec:so02q}, we know that we can fix a Weyl alcove 
$$\mathcal{A}=\{T(2\pi\alpha,2\pi\beta_i)\in\mathfrak{t}\mid0<\beta_i\pm\beta_j<1,\forall1\leqslant i<j\leqslant n\}$$
of $H=\SO(2)\times\SO(q)$ and choose parabolic weights $$\tau^j\in\dfrac{\iu}{2\pi}\overline{\mathcal{A}}=\{T(\iu\alpha^j,\iu\beta_i^j)\in\mathfrak{t}\mid0\leqslant\beta_k\pm\beta_l\leqslant1,\forall1\leqslant k<l\leqslant n\}$$ for the corresponding marked points $x_j$. This induces us to define the concept of $\SO_0(2,q)$-weight.

\begin{definition}\label{defn:so02qweight}
    An \textbf{$\SO_0(2,q)$-weight (over $(X,D)$)} is defined as a $(q+2)\times s$-tuple $$({\alpha},{\beta}):=(\alpha^j,-\alpha^j,\beta_i^j)_{1\leqslant i\leqslant q,1\leqslant j\leqslant s},$$
    such that 
    \begin{itemize}
        \item[(1)] $\beta_i^j+\beta_{q+1-i}^j=0$ for any $1\leqslant i\leqslant q$;

        \item[(2)] $\beta_k^j+\beta_l^j\in[-1,1]$ for any $1\leqslant k,l\leqslant q$;

        \item[(3)] $\beta_i^j$ is non-increasing with respect to $i$.
    \end{itemize}
\end{definition}
For an $\SO_0(2,q)$-weight $(\alpha,\beta)$, we use $\mathcal{M}(\alpha,\beta)$ to denote the moduli space of equivalence classes of polystable parabolic $\SO_0(2,q)$-Higgs bundles over $(X,D)$ with parabolic weights $\tau^j=T(\iu\alpha^j,\iu\beta^j)$ introduced in \cite[Section 7.1]{biquard2020parabolic}.

\begin{remark}
    $\mathcal{M}(\alpha,\beta)$ coincides with the S-equivalence classes of semistable parabolic $\SO_0(2,q)$-Higgs bundles over $(X,D)$ with parabolic weights $(\alpha,\beta)$. We recommend the readers to see \cite{schmitt2008geometric} for a reference and a summary from non-parabolic, complex reductive group case to parabolic, real reductive group case. We also refer to \cite{kydonakis2023logahoric} for parahoric, complex reductive case.
\end{remark} 

In this section we will prove the following proposition.

\begin{proposition}\label{prop:vbofSO02q}
    A parabolic $\SO_0(2,q)$-Higgs bundle $(\mathbb{E},\Phi)$ of weight $\tau^j=T(\iu\alpha^j,\iu\beta^j)$, where $(\alpha,\beta)$ is an $\SO_0(2,q)$-weight, over $(X,D)$ is equivalent to the following data:
    \begin{itemize}
        \item[(1)] a holomorphic vector bundle $\EE=\mathcal{L}^{\vee}\oplus\mathcal{L}\oplus\mathcal{V}$ over $X$, where $\mathcal{L}$ is a holomorphic line bundle, $\operatorname{rank}(\mathcal{V})=q$ and $\det(\mathcal{V})\cong\mathcal{O}$ which is the trivial line bundle;
    
        \item[(2)] a holomorphic symmetric non-degenerate bilinear form $Q_\VV\colon\operatorname{Sym}^2(\VV)\to\mathcal{O}$ with the induced isomorphism $q_\VV\colon\VV\to\VV^\vee$; 
    
        \item[(3)] a reverse isotropic flag $\VV^j=\left(\VV_t^j\right)$ of $\VV_{x_j}$ with respect to $Q_\VV$ monotonically weighted by $\beta^j:=\left(\beta_l^j\right)_{1\leqslant l\leqslant q}$ for every marked point $x_j\in D$; 
        
        \item[(4)]
        a Higgs field $\Phi=\begin{pmatrix}
            0&0&\eta\\
            0&0&\gamma\\
            -\gamma^*&-\eta^*&0
        \end{pmatrix}$ with respect to the decomposition $\EE=\LL^\vee\oplus\LL\oplus\VV$ which is a global holomorphic section of $\operatorname{End}(\mathcal{E})\otimes\KK(D)$, where $\eta$ and $\gamma$ are sections of $\operatorname{Hom}(\VV,\LL^\vee)\otimes\KK(D)$ and $\operatorname{Hom}(\VV,\LL)\otimes\KK(D)$ respectively, \[\eta^*=q_{\VV}^{-1}\circ\eta^{\vee},\quad \gamma^*=q_{\VV}^{-1}\circ\gamma^{\vee},\]
        such that
        \[\eta=\sum_{l=1}^qO\left(z^{\left\lceil\alpha^j-\beta_l^j\right\rceil}\right)\cdot e^\vee\otimes f_l^\vee\otimes\dfrac{\dd z}{z},\quad \gamma=\sum_{l=1}^qO\left(z^{\left\lceil-\alpha^j-\beta_l^j\right\rceil}\right)\cdot e\otimes f_l^\vee\otimes\dfrac{\dd z}{z}\]
        on a coordinate chart $(U,z)$ centered at $x_j$ for every $x_j\in D$, where $e$ (resp. $e^\vee$) is a holomorphic frame of $\LL$ (resp. $\LL^{\vee}$), $\{f_1,\dots,f_q\}$ is a holomorphic frame compatible with the reverse isotropic flag $\left(\VV_i^j\right)$.
    \end{itemize}
    
\end{proposition}

We also use $(\mathcal{E},\Phi)$ or $(\LL^{\vee}\oplus\LL\oplus\VV,\Phi)$ to denote a parabolic $\SO_0(2,q)$-Higgs bundle instead of $(\mathbb{E},\Phi)$ sometimes. Besides, we usually use $Q_\UU:=\begin{pmatrix}
            0&1\\1&0
        \end{pmatrix}$ to denote the natural pairing on $\UU:=\LL^\vee\oplus\LL$.

\begin{proof}

We fix a principal $H^\mathbb{C}$-bundle $\mathbb{E}$, where $H^{\mathbb{C}}=\SO(2,\mathbb{C})\times\SO(q,\mathbb{C})$, we consider the associated vector bundle $\mathcal{E}:=\mathbb{E}\times_{H^\mathbb{C}}(\mathbb{C}^2\oplus\mathbb{C}^q)=\mathcal{U}\oplus\mathcal{V}$ via the standard representation. It is well-known that $\EE$ has the desired decomposition $\LL^\vee\oplus\LL\oplus\VV$ where $\VV$ admits an orthogonal structure $Q_\VV$, c.f. \cite[Definition 2.5]{aparicio2019so} and \cite[Proposition 2.14]{collier2019geometry}. Therefore, from the discussion on the parabolic subgroup $P_{\tau^j}$ above and the definition of a parabolic principal $\SO(2,\mathbb{C})\times\SO(q,\mathbb{C})$-bundle, c.f. \cite[Section 2.1]{biquard2020parabolic}, we obtain (1)(2)(3). 

To determine what a parabolic $\mathrm{SO}_0(2,q)$-Higgs field is, we have to examine what the sheaf $P\mathbb{E}(\mathfrak{m}^\mathbb{C})$ of parabolic sections of $\mathbb{E}(\mathfrak{m}^\mathbb{C})$ is, where $\mathfrak{m}^{\mathbb{C}}$ in complexified Cartan decomposition is characterized in \prettyref{sec:so02q}. We still use the basis $\mathcal{B}'$ to diagonalize $\tau^j$, and under this basis $\mathfrak{m}^{\mathbb{C}}$ has a basis 
$$\{e_{k,l}=E_{k,l+2}-E_{q+3-l,3-k}\mid k\in\{1,2\},1\leqslant l\leqslant q\},$$ where $(E_{k,l})_{i,j}=\delta_{i,k}\cdot\delta_{j,l}$ for the Kronecker symbol $\delta$. Indeed, $e_{k,l}$ is the eigenvector of $\operatorname{ad}(\tau^j)$ acting on $\mathfrak{m}^{\mathbb{C}}$ with eigenvalue 
$$\mu_{k,l}=\begin{cases}
    \alpha^j-\beta_{l}^j& \mbox{when }k=1,\\
    -\alpha^j-\beta_{l}^j& \mbox{when }k=2.
\end{cases}$$

Therefore, suppose a meromorphic (around $x_j$) section $\phi$ of $\mathbb{E}(\mathfrak{m}^\mathbb{C})$ can be represented as 

$$\begin{pmatrix}
    0&0&\eta\\
    0&0&\gamma\\
    -\gamma^*&-\eta^*&0
\end{pmatrix}$$
under the decomposition $\LL^{\vee}\oplus\LL\oplus\VV$ for meromorphic (around $x_j$) sections $\eta,\gamma$ of $\operatorname{Hom}(\VV,\LL^{\vee})$ and $\operatorname{Hom}(\VV,\LL)$ respectively, where
$$\eta^*=q_{\VV}^{-1}\circ\eta^{\vee},\quad \gamma^*=q_{\VV}^{-1}\circ\gamma^{\vee},$$
then $\phi$ is a section of $P\mathbb{E}(\mathfrak{m}^\mathbb{C})$ if and only if with respect to a coordinate chart $(U,z)$ centered at $x_j$, a holomorphic frame $e$ (resp. $e^\vee$) of $\LL$ (resp. $\LL^{\vee}$), a holomorphic frame $\{f_1,\dots,f_q\}$ compatible with the reverse isotropic flag $\left(\VV_i^j\right)$, \[\eta=\sum_{l=1}^qO\left(z^{\left\lceil\alpha^j-\beta_l^j\right\rceil}\right)\cdot e^\vee\otimes f_l,\quad \gamma=\sum_{l=1}^qO\left(z^{\left\lceil-\alpha^j-\beta_l^j\right\rceil}\right)\cdot e\otimes f_l.\]
Now (4) follows by the definition of $\Phi$ and above discussions.
\end{proof}

There is a continuous map
$\mathbf{d}\colon\mathcal{M}(\alpha,\beta)\longrightarrow\mathbb{Z}$ which sends $\left[(\LL^{\vee}\oplus\LL\oplus\VV,\Phi)\right]$ to $\deg(\LL)$. Hence $\mathcal{M}(\alpha,\beta)$ can be decomposed into $\coprod_{d\in\mathbb{Z}}\mathcal{M}(\alpha,\beta,d)$, where $\mathcal{M}(\alpha,\beta,d):=\mathbf{d}^{-1}(d)$.

\subsection{Stability of parabolic \texorpdfstring{$\mathrm{SO}_0(2,q)$}{SO0(2,q)}-Higgs bundles}

In this subsection, we will give the stability condition of parabolic $\mathrm{SO}_0(2,q)$-Higgs bundles via the vector bundle viewpoint. See \cite[Section 4.2]{biquard2020parabolic} for the definition of stability condition of parabolic $G$-Higgs bundles. 

\begin{restatable}{proposition}{stability}\label{prop:stability}
    A parabolic $\SO_0(2,q)$-Higgs bundle $(\EE=\LL^{\vee}\oplus\LL\oplus\VV,\Phi)$ with weight $\tau^j$ at $x_j$ is semistable if and only if $\operatorname{pardeg}(\UU')+\operatorname{pardeg}(\VV')\leqslant 0$ for any isotropic subbundles $\UU'\subset\LL^{\vee}\oplus\LL$, $\VV'\subset\VV$ satisfying $\UU'\oplus\VV'$ is $\Phi$-invariant. Moreover, $(\EE,\Phi)$ is stable if and only if the above inequality is strict when $\VV'$ is nonzero.
    
\end{restatable}

We first do some preparation. Let $P\subset\SO(q,\mathbb{C})$ be a parabolic subgroup, $\chi$ an antidominant character of $P':=\SO(2,\mathbb{C})\times P$
and
$\sigma\in\mathrm{H}^0(X,\mathbb{E}(\SO(q,\mathbb{C})/P))$ a holomorphic reduction. Suppose that 
$\rho_V\colon\SO(2)\times\SO(q)\to\mathrm{U}(2+q)$ is the standard representation and it can be naturally extended to the standard representation 
$\rho_V\colon\SO(2,\mathbb{C})\times\SO(q,\mathbb{C})\to\mathrm{GL}(2+q,\mathbb{C})$ acting on the complex linear space $V=\mathbb{C}^{2+q}$.
Let $s_\chi$ be the dual element of $\chi$ with respect to the invariant bilinear form. Then $\mathrm{d}\rho_V(s_{\chi})$ can be diagonalized with real eigenvalues $$-\lambda,\,\lambda,\,\mu_1<\mu_2<\cdots<\mu_{t'},\,\mu_k=-\mu_{t'+1-k}.$$
We rearrange them as 
$
\lambda_1<\lambda_2<\cdots<\lambda_{t''}$ and assume that  $\lambda_{t''+1}=0$. Define $V_j=\operatorname{ker}(\lambda_j\operatorname{id}_V-\mathrm{d}\rho_V(s_{\chi}))$ and $W_j:=\bigoplus_{k=1}^j V_k$. We obtain holomorphic vector bundles $\mathcal{W}_j=\mathbb{E}(W_j)$ and this gives the filtration
\begin{equation}\label{filtration}
    0=\mathcal{W}_0\subset\mathcal{W}_1\subset\cdots\mathcal{W}_{t''}=\EE,
\end{equation}
Similar to the $\mathrm{GL}(n,\mathbb{C})$ case (see (\cite[Appendix A]{kydonakis2021topological})), we obtain that 
\begin{equation}\label{pardeg}
    \operatorname{pardeg}\mathbb{E}(\sigma,\chi)=\sum_{k=1}^{t''}(\lambda_{k}-\lambda_{k+1})\operatorname{pardeg}(\mathcal{W}_{k}),
\end{equation}
where $\operatorname{pardeg}\mathbb{E}(\sigma,\chi)$ denotes the parabolic degree of the reduction $\sigma$ with respect to the antidominant character $\chi$. One can see \cite[Section 2.2]{biquard2020parabolic} or \cite[Definition A.15]{kydonakis2021topological} for its definition. Denote
$\mathfrak{m}_{\tau}:=\left\{v\in\mathfrak{m}^{\mathbb{C}}\mid \operatorname{Ad}(\exp(t\tau))v\mbox{ is bounded when }t\to\infty\right\}$
for any parabolic subgroup $P_\tau$ and $\mathbb{E}_\sigma\left(\mathfrak{m}_\tau\right):=\left(\sigma^*\mathbb{E}\right)\times_{P_{\tau}}\mathfrak{m}_{\tau}.$

To prove \prettyref{prop:stability} of parabolic $\SO_0(2,q)$-Higgs bundle via subbundles, we need to prove the following lemma.

\begin{lemma}\label{lemma:pardegoforthocomplement}
Suppose $\VV'$ is an isotropic subbundle of $\VV$ equipped with the induced parabolic structure from $\VV$, then $\deg(\VV')=\deg\left((\VV')^{\perp}\right)$ and $\operatorname{pardeg}(\VV')=\operatorname{pardeg}\left((\VV')^{\perp}\right).$
\end{lemma}

\begin{proof}
    Since $\VV'$ is isotropic, the non-degenerate bilinear form $Q_{\VV}$ descends to the quotient bundle $(\VV')^{\perp}/\VV'$. Hence $\deg(\VV')=\deg\big((\VV')^{\perp}\big)$.
    
    Now we focus on the parabolic part. Suppose that the isotropic flag of $\VV$ at $x_j$ is $$0=\VV_{t+1}^j\subset\VV_t^j\subset\cdots\subset\VV_1^j=(\VV)_{x_j}$$
    with weight $\tilde{\beta}_{1}^j>\tilde{\beta}_{2}^j>\cdots>\tilde{\beta}_{t}^j$ and assume $\tilde{\beta}_0^j=\tilde{\beta}_{t+1}^j=0$. 
    Then 
    $$\begin{aligned}
        &\sum_{l=1}^{t}\left(\tilde{\beta}_{t+1-l}^j-\tilde{\beta}_{t-l}^j\right)\dim\left(\VV_{t+1-l}^j\cap(\VV')^{\perp}_{x_j}\right)\\
        =&\sum_{l=1}^{t}\left(\tilde{\beta}_{t+1-l}^j-\tilde{\beta}_{t-l}^j\right)\left(\dim(\VV)_{x_j}-\dim\left(\VV_{l+1}^j\oplus(\VV')_{x_j}\right)\right)\quad\hspace{-1em}\mbox{(since $U^\perp\cap V^\perp=(U\oplus V)^\perp$)}\\
        =&\tilde{\beta}_t^jq-\sum_{l=1}^{t}\left(\tilde{\beta}_{t+1-l}^j-\tilde{\beta}_{t-l}^j\right)\left(\dim\left(\VV_{l+1}^j\right)+\dim\left((\VV')_{x_j}\right)-\dim\left(\VV_{l+1}^j\cap(\VV')_{x_j}\right)\right)\\
        =&\tilde{\beta}_t^j\left(q-\dim\left((\VV')_{x_j}\right)\right)-\sum_{l=1}^{t}\left(\tilde{\beta}_{t+1-l}^j-\tilde{\beta}_{t-l}^j\right)\left(\dim\left(\VV_{l+1}^j\right)\right)\\&\quad\quad\quad\quad\quad\quad\quad\quad\quad\quad\quad\quad\quad\quad\quad\quad\quad\quad+\sum_{l=1}^{t}\left(\tilde{\beta}_{t+1-l}^j-\tilde{\beta}_{t-l}^j\right)\dim\left(\VV_{l+1}^j\cap(\VV')_{x_j}\right).
    \end{aligned}$$
    Now we set 
    \[I:=\tilde{\beta}_t^j\left(q-\dim\left((\VV')_{x_j}\right)\right)-\sum_{l=1}^{t}\left(\tilde{\beta}_{t+1-l}^j-\tilde{\beta}_{t-l}^j\right)\left(\dim\left(\VV_{l+1}^j\right)\right).\]
    Then the above formula equals to
    $$\begin{aligned}
        &I+\sum_{l=1}^{t}\left(\tilde{\beta}_{t+1-l}^j-\tilde{\beta}_{t-l}^j\right)\dim\left(\VV_{l+1}^j\cap(\VV')_{x_j}\right)\\
        =&I+\sum_{l=1}^{t}\left(\tilde{\beta}_{l}^j-\tilde{\beta}_{l-1}^j\right)\dim\left(\VV_{t+2-l}^j\cap(\VV')_{x_j}\right)\quad(\mbox{change the index from $l$ to $t+1-l$})\\
        =&I+\sum_{l=1}^{t}\left(\tilde{\beta}_{t+2-l}^j-\tilde{\beta}_{t+1-l}^j\right)\dim\left(\VV_{t+2-l}^j\cap(\VV')_{x_j}\right)\quad(\mbox{recall that $\tilde{\beta}_{l}^j+\tilde{\beta}_{t+1-l}^j=0$})\\
        =&I+\sum_{l=1}^{t}\left(\tilde{\beta}_{t+1-l}^j-\tilde{\beta}_{t-l}^j\right)\dim\left(\VV_{t+1-l}^j\cap(\VV')_{x_j}\right)-\tilde{\beta}_1^j\dim\left((\VV')_{x_j}\right),
    \end{aligned}$$
    so it suffices to show that
    $$I-\tilde{\beta}_1^j\dim\left((\VV')_{x_j}\right)=\tilde{\beta}_t^jq-\sum_{l=1}^{t}\left(\tilde{\beta}_{t+1-l}^j-\tilde{\beta}_{t-l}^j\right)\left(\dim\left(\VV_{l+1}^j\right)\right)=0.$$
    Actually, 
    $$\begin{aligned}
        &\tilde{\beta}_t^jq-\sum_{l=1}^{t}\left(\tilde{\beta}_{t+1-l}^j-\tilde{\beta}_{t-l}^j\right)\left(\dim\left(\VV_{l+1}^j\right)\right)\\
        =&\tilde{\beta}_t^jq-\sum_{l=1}^{t}\left(\tilde{\beta}_{t+1-l}^j-\tilde{\beta}_{t-l}^j\right)\left(q-\dim\left(\VV_{t+1-l}^j\right)\right)\\
        =&\sum_{l=1}^{t}\left(\tilde{\beta}_{t+1-l}^j-\tilde{\beta}_{t-l}^j\right)\dim\left(\VV_{t+1-l}^j\right)
    \end{aligned}$$
    is the parabolic part of $\operatorname{pardeg}(\VV)$, indeed $0$.
\end{proof}

\begin{proof}[Proof of \prettyref{prop:stability}]
We first show that if $(\EE,\Phi)$ is semistable, then $\operatorname{pardeg}(\UU')+\operatorname{pardeg}(\VV')\leqslant 0$. Consider the isotropic flag $$0\subset\VV'\subset(\VV')^{\perp}\subset\VV,$$
it means that $\VV$ can be viewed as a parabolic principal $P_\tau$-bundle with a holomorphic reduction $\sigma\in\mathrm{H}^0(X,\mathbb{E}(\SO(q,\mathbb{C})/P_\tau))$ for the parabolic group $P_\tau\subset\SO(q,\mathbb{C})$ which preserves the flag of the form above in $\mathbb{C}^q$. Denote $\operatorname{rank}(\VV')$ by $v$. Note that any strictly antidominant element $s_\chi$ of $P_\tau':=\SO(2,\mathbb{C})\times P_\tau$ can be diagonalized as $$\operatorname{diag}(\lambda,-\lambda,\mu_1,\dots,\mu_q)$$
with $\lambda\in\mathbb{R}$, $\mu_1=\cdots=\mu_{v}=:\mu>0$ when $v>0$, and we set $\mu_1=\cdots=\mu_q$=0 when $v=0$. We have $\mu_{q+1-k}=-\mu_k$ for any $1\leqslant k\leqslant v$ and $\mu_k=0$ otherwise. If $\operatorname{rank}(\UU')=1$, without loss of generality, we assume that $\UU'=\LL$, we take that $\lambda=\mu$. Then $\LL\oplus\VV'$ is $\Phi$-invariant shows that $\Phi|_{X\setminus D}\in\mathrm{H}^0\left(X\setminus D,\mathbb{E}_\sigma\left(\mathfrak{m}_\tau\right)\otimes\KK(D)\right).$

Now the filtration $(\mathcal{W}_k)$ corresponds to $\sigma,\chi$ (defined in (\ref{filtration})) is
\[\begin{tikzcd}
	0 & {\mathcal{L}\oplus\mathcal{V}^\prime} & {\mathcal{L}\oplus(\mathcal{V}^\prime)^\perp} & {\mathcal{E}} \\
	& {\mathcal{W}_1} & {\mathcal{W}_2} & {\mathcal{W}_3}
	\arrow["{=}"{marking}, draw=none, from=1-2, to=2-2]
	\arrow["{=}"{marking}, draw=none, from=1-3, to=2-3]
	\arrow["{=}"{marking}, draw=none, from=1-4, to=2-4]
	\arrow["\subset"{marking}, draw=none, from=1-1, to=1-2]
	\arrow["\subset"{marking}, draw=none, from=1-2, to=1-3]
	\arrow["\subset"{marking}, draw=none, from=1-3, to=1-4]
	\arrow["\subset"{marking, pos=0.4}, draw=none, from=2-2, to=2-3]
	\arrow["\subset"{marking, pos=0.7}, draw=none, from=2-3, to=2-4]
\end{tikzcd}\]
by definition. Hence by (\ref{pardeg}) we obtain that
\begin{equation}\label{stable1}
    \begin{aligned}
    0\leqslant&\operatorname{pardeg}\mathbb{E}(\sigma,\chi)\\
    =&-\lambda\operatorname{pardeg}(\LL\oplus\VV')-\lambda\operatorname{pardeg}(\LL\oplus(\VV')^{\perp})+\lambda\operatorname{pardeg}(\EE)\\
    =&-2\lambda\left(\operatorname{pardeg}(\LL)+\operatorname{pardeg}(\VV')\right),
\end{aligned}
\end{equation}
and this implies that $\operatorname{pardeg}(\LL)+\operatorname{pardeg}(\VV')\leqslant0$. If $\UU'=0$, we take $\lambda=0$ and $\mu>0$, then $\Phi(\VV')=0$ shows that $\Phi|_{X\setminus D}\in\mathrm{H}^0\left(X\setminus D,\mathbb{E}_\sigma\left(\mathfrak{m}_\tau\right)\otimes\KK(D)\right)$
and then similarly as above we obtain that
\begin{equation}\label{stable2}
\begin{aligned}
    0\leqslant&\operatorname{pardeg}\mathbb{E}(\sigma,\chi)\\
    =&-\mu\operatorname{pardeg}(\VV')-\mu\operatorname{pardeg}(\LL\oplus\LL^{\vee}\oplus(\VV')^{\perp})+\mu\operatorname{pardeg}(\EE)\\
    =&-2\mu\operatorname{pardeg}(\VV').
\end{aligned}\end{equation}
This shows that $\operatorname{pardeg}(\VV')\leqslant0$.

Now we assume that $\operatorname{pardeg}(\UU')+\operatorname{pardeg}(\VV')\leqslant 0$ for any isotropic subbundle $\UU'\subset\LL^{\vee}\oplus\LL$, $\VV'\subset\VV$ satisfying $\UU'\oplus\VV'$ is $\Phi$-invariant. For any parabolic subgroup $P_\tau'\subset\SO(2,\mathbb{C})\times\SO(q,\mathbb{C})$, any antidominant character $\chi$, any holomorphic reduction $\sigma$, we know that if $s_{\chi}$ can be diagonalized with eigenvalues $\lambda_1<\lambda_2<\cdots<\lambda_{t''}$, then by (\ref{pardeg}) again, 
\begin{equation}\label{stable3}
    \operatorname{pardeg}\mathbb{E}(\sigma,\chi)=\sum_{k=1}^{t''}(\lambda_{k}-\lambda_{k+1})\operatorname{pardeg}(\mathcal{W}_{k}).
\end{equation}
with assuming $\lambda_{t^\dprime+1}=0$.
Note that
$\Phi|_{X\setminus D}\in\mathrm{H}^0\left(X\setminus D,\mathbb{E}_\sigma\left(\mathfrak{m}_\tau\right)\otimes\KK(D)\right)$
means that every $\mathcal{W}_k$ is $\Phi$-invariant. If $k\leqslant\lfloor t''/2\rfloor$, $\mathcal{W}_k$ is isotropic and it splits as $\UU'\oplus\VV'$ which satisfies the conditions above, so $\operatorname{pardeg}(\mathcal{W}_k)\leqslant0$. If $k>\lfloor t''/2\rfloor$, we can also get $\operatorname{pardeg}(\mathcal{W}_k)\leqslant0$ due to \prettyref{lemma:pardegoforthocomplement}. Hence we get $\operatorname{pardeg}(\mathcal{W}_k)\leqslant0$ for every $k$ and automatically we have $\operatorname{pardeg}(\mathcal{W}_{t^\dprime})=0$. Therefore, $\operatorname{pardeg}\mathbb{E}(\sigma,\chi)\geqslant0$ due to $\lambda_1<\lambda_2<\cdots<\lambda_{t^\dprime}$.

In addition, if $(\mathcal{E},\Phi)$ is stable, then for any $\sigma,\chi$, we have $\operatorname{pardeg}\mathbb{E}(\sigma,\chi)>0$. Then for any proper $\mathcal{V}'$, we have $\mu>0$ and the inequalities in (\ref{stable1}) and (\ref{stable2}) are strict, so $\operatorname{pardeg}(\UU')+\operatorname{pardeg}(\VV')<0$. Conversely, we assume $\operatorname{pardeg}(\UU')+\operatorname{pardeg}(\VV')<0$ for every proper $\VV'$. For any $$s_\chi\in\iu\left(\left(\mathfrak{h}\cap\mathfrak{z}(\mathfrak{g})\right)^{\perp_B}\setminus\left(\mathfrak{h}\cap\mathfrak{z}(\mathfrak{g})\right)\right)=\iu\mathfrak{h}\setminus\{0\}\mbox{ for }\mathfrak{g}=\mathfrak{so}(2,q), \mathfrak{h}=\mathfrak{so}(2)\oplus\mathfrak{so}(q),$$ the filtration $(\mathcal{W}_k)$ we get in (\ref{stable3}) is nontrivial, i.e. $t^\dprime>1$. Hence $\operatorname{pardeg}\mathbb{E}(\sigma,\chi)>0$ follows.
\end{proof}

\subsection{Parabolic \texorpdfstring{$\mathrm{SO}(n,\mathbb{C})$}{SO(n,C)}-Higgs bundles and stability}

Similar to the above discussion, we can apply the general theory to $G=\mathrm{SO}(n,\mathbb{C})$ for $n>2$. We omit the proof in this subsection. Through the vector bundle viewpoint, we can get a parabolic $\mathrm{SO}(n,\mathbb{C})$-Higgs bundle is equivalent to the following data:
\begin{itemize}
    \item[(1)] A rank $n$ vector bundle $\mathcal{E}$ with a non-degenerate symmetric bilinear form $Q_{\mathcal{E}}\colon\mathcal{E}\otimes\mathcal{E}\to\mathcal{O}$ on $\mathcal{E}$, i.e. it induces an isomorphism $q_{\mathcal{E}}\colon\mathcal{E}\to\mathcal{E}^{\vee}$. Moreover, $\det(\mathcal{E})\cong\mathcal{O}$;

    \item[(2)] The parabolic structure corresponds to reverse isotropic flags at each $x_j$;

    \item[(3)] A parabolic Higgs field $\Phi$ satisfying $q_{\mathcal{E}}^{-1}\circ\Phi^{\vee}\circ q_{\mathcal{E}}=-\Phi$.
\end{itemize}

Therefore, one can view a parabolic $\SO_0(2,q)$-Higgs bundle as a parabolic $\SO(2+q,\mathbb{C})$-Higgs bundle naturally.

For stability, we can get
\begin{proposition}\label{prop:stabsonC}
    For $n>2$, a parabolic $\SO(n,\mathbb{C})$-Higgs bundle $(\mathcal{E},\Phi)$ is semistable if and only if for any $\Phi$-invariant isotropic subbundle $\mathcal{E}^\prime\subset\mathcal{E}$, $\operatorname{pardeg}(\mathcal{E}^\prime)\leqslant 0$. And it is stable if and only if the above inequality is strict when $\mathcal{E}^\prime$ is nonzero.
\end{proposition}

Since the bilinear form is non-degenerate, $\EE$ itself cannot be isotropic. So here the $\Phi$-invariant isotropic subbundle could not be $\EE$ itself.

\begin{remark}\label{rem:diferrentstability}
    We need to note that there exist some stable parabolic $\SO_0(2,q)$-Higgs bundles which are not stable as parabolic $\SO(2+q,\mathbb{C})$-Higgs bundles from \prettyref{prop:stability} and \prettyref{prop:stabsonC}. The main reason of this difference is $\mathfrak{so}(2)\oplus\mathfrak{so}(q)$ has $\mathfrak{so}(2)$ as a nontrivial center,  hence it does not give any reduction on $\LL^\vee\oplus\LL$. In \cite[Definition 2.13]{tholozan2021compact}, we can see that the stability they defined for their parabolic $\mathrm{SU}(p,q)$-Higgs bundles coincides with our usual $\mathrm{SL}(p+q,\mathbb{C})$-stability.
\end{remark}

\subsection{Some auxiliary results}\label{section:NAH}

We first introduce the Hitchin fibration. Its properness will be the most important tool to deduce the compactness of some connected components.

\begin{definition}
    The \textbf{Hitchin fibration} is defined as
    $$\begin{aligned}
        \Pi_{Hit}\colon \mathcal{M}(\alpha,\beta)&\longrightarrow\bigoplus_{i=1}^{\lfloor q/2\rfloor+1}\mathrm{H}^0(X,\mathcal{K}(D)^{2i})\\ \big[(\mathcal{E},\Phi)\big]&\longmapsto\big(\operatorname{tr}(\Phi^{2i})\big)_{i=1}^{\lfloor q/2\rfloor+1}.
    \end{aligned}$$
\end{definition}

It is well-known that:

\begin{proposition}\label{prop:hitchinproper}
    $\Pi_{Hit}$ is proper, i.e. the preimage of a compact subset is compact.
\end{proposition}

\begin{remark}
    The properness of the Hitchin fibration was proven by Hitchin for $\mathrm{GL}(n,\mathbb{C})$-Higgs bundles over closed Riemann surfaces in \cite{Hitchin1987StableBA} and extended by Yokogawa to general parabolic $\mathrm{GL}(n,\mathbb{C})$-Higgs sheaves in \cite[Section 5]{Yokogawa1993CompactificationOM}. 
    
    For general parabolic $G$-Higgs bundles, when $G$ is an algebraic group over any algebraically closed field, the properness is proven by Yun in \cite[Corollary 2.5.2]{yun2011global}. Moreover, when $G$ is a real reductive Lie group, the properness follows by identifying the moduli space of parabolic $G$-Higgs bundles with a closed subset consisting of the fixed points of Cartan involution in the moduli space of $G^{\mathbb{C}}$-Higgs bundles, up to a finite-sheeted covering (see \cite[Proposition 5.9]{garcia2019involutions}).
\end{remark}

To connect the representation side and Higgs bundle side, we introduce the non-abelian Hodge correspondence here. We fix a closed Riemann surface $X$ of genus $g$ with $s$ marked points $x_1,\dots,x_s$ on it, and we regard $X\setminus D$ as $\Sigma_{g,s}$. Let $(\alpha,\beta)$ be an $\SO_0(2,q)$-weight and $\tau=(\tau^j)_{j=1}^s$ be the corresponding weight in $\dfrac{\iu}{2\pi}\mathcal{A}$. Define 
$h(\alpha,\beta):=\left(\exp(2\pi\iu\cdot\tau^j)\right)_{j=1}^s.$ Recall that a parabolic $\SO_0(2,q)$-Higgs bundle $(\mathbb{E},\Phi)$ is called \textbf{simple} if the automorphism group $\operatorname{Aut}(\mathbb{E},\Phi)$ is $Z(\SO(2+q,\mathbb{C}))\cap\operatorname{ker}\iota$, where $\iota$ denotes the isotropy representation.

The non-abelian Hodge correspondence (for structure group $\SO_0(2,q)$) says that
\begin{proposition}\label{prop:NAH}
    For any $\SO_0(2,q)$-weight $(\alpha,\beta)$ such that $\alpha^j-\beta_i^j\notin\mathbb{Z}$ for any $i,j$, there exists a homeomorphism $$\mathsf{NAH}\colon\mathcal{M}(\alpha,\beta)\longrightarrow\mathfrak{X}_{h(\alpha,\beta)}\big(\Sigma_{g,s},\mathrm{SO}_0(2,q)\big).$$ Through this correspondence, stable, simple Higgs bundles, which are also stable as parabolic $\SO(2+q,\mathbb{C})$-Higgs bundle, are mapped into irreducible representations.
\end{proposition}

\begin{remark}
    The non-abelian Hodge correspondence in the parabolic setting was first established by Simpson for $\mathrm{GL}(n,\mathbb{C})$ by C. T. Simpson in \cite{Simpson1990HarmonicBO} and later extended to arbitrary real reductive Lie groups by O. Biquard, O. García-Prada, and I. Mundet i Riera in \cite{biquard2020parabolic}. In our case, the relationship between parabolic weights and monodromy is straightforward. However, complications arise when some eigenvalues of the isotropy representation (such as $\pm\alpha^j-\beta_i^j$'s for $\mathrm{SO}_0(2,q)$ case) take integer values. In such a case, the graded residue of the Higgs field must be involved, and the correspondence involves an additional unipotent element associated with it.
    
    Furthermore, if the graded residue of the Higgs field is non-zero, the Betti moduli space and the character variety generally fail to be isomorphic. In particular, there exist parabolic Higgs bundles corresponding to indecomposable representations. For a comprehensive treatment of the correspondence in the parabolic case, we refer to \cite[Table 1]{biquard2020parabolic}.
\end{remark}

By using bounded cohomology, the Toledo invariant was introduced in \cite{burger2010surface}. Since $\SO_0(2,q)$ is of Hermitian type, \textbf{Toledo invariant} is a function $\operatorname{Tol}\colon \mathfrak{X}(\Sigma_{g,s},\SO_0(2,q))\to\mathbb{R}$. In \cite[Theorem 1]{burger2010surface}, M. Burger, A. Iozzi, and A. Wienhard showed the following results.

\begin{proposition}\label{prop:tol1}
    $\operatorname{Tol}$ is continuous.
\end{proposition}

\begin{proposition}\label{prop:tol2}
    If $b$ is a simple closed curve on $\Sigma_{g,s}$ separating it into two subsurfaces $\Sigma^\prime$ and $\Sigma^\dprime$, then for every representation $\rho\colon\pi_1\left(\Sigma_{g,s}\right)\to\SO_0(2,q)$, 
    $$\operatorname{Tol}(\rho)=\operatorname{Tol}(\rho|_{\pi_1(\Sigma^\prime)})+\operatorname{Tol}(\rho|_{\pi_1(\Sigma^\dprime)})$$
\end{proposition}

We will denote by $\mathfrak{X}_h^{\tau}\big(\Sigma_{g,s},\SO_0(2,q)\big):=\operatorname{Tol}^{-1}(\tau)$. We will sometimes call this space a \textbf{relative component}, though it may not be connected in general.

Through $\mathsf{NAH}$, one can say something about the Toledo invariant of a parabolic $\SO_0(2,q)$-Higgs bundle. But it can also be constructed directly from the \textbf{Toledo character}, for more details, see \cite{biquard2017higgs} for the closed surface case and \cite{biquard2020parabolic} for the parabolic case. Under our setting, we have that

\begin{proposition}
    For a parabolic $\SO_0(2,q)$-Higgs bundle $(\LL^\vee\oplus\LL\oplus\VV,\Phi)$, 
    $$\operatorname{Tol}\left(\mathsf{NAH}\left(\left[(\LL^\vee\oplus\LL\oplus\VV,\Phi)\right]\right)\right)=\operatorname{pardeg}(\LL),$$
    up to multiplying a constant.
\end{proposition}

Therefore, we can restrict $\mathsf{NAH}$ to $\mathsf{NAH}\colon\mathcal{M}(\alpha,\beta,d)\to\mathfrak{X}_{h(\alpha,\beta)}^{d+|\alpha|}\big(\Sigma_{g,s},\SO_0(2,q)\big)$.

\section{Compact components in \texorpdfstring{$\mathcal{M}(\alpha,\beta)$}{M(α,β)}}\label{sec:cc}

In this section, we assume that $X=\mathbb{C}P^1$ is the complex projective line. We will consider a parabolic $\SO_0(2,q)$-Higgs bundle $(\EE=\LL^{\vee}\oplus\LL\oplus\VV,\Phi)$ over $(X,D)$ with non-degenerate bilinear form $Q_{\VV}$ on $\VV$, weight $\tau=(\tau^j)$ corresponds to the $\SO_0(2,q)$-weight $(\alpha,\beta)$ at $x_j$, that is, $\tau^j=T(\iu\alpha^j,\iu\beta_i^j)$, where $T$ is defined in \prettyref{eq:Cartan},
and the Higgs field $$\Phi=\begin{pmatrix}
    0&0&\eta\\
    0&0&\gamma\\
    -\gamma^*&-\eta^*&0
\end{pmatrix}.$$

We can add some divisors to $\LL$ to get an equivalent parabolic $\SO_0(2,q)$-Higgs bundle whose weights $\alpha^j\in[-1/2,1/2]$. If we assume $\alpha^j>\beta_1^j$ in addition, then the parabolic weights $\tau^j=T(\iu\alpha^j,\iu\beta_i^j)$ lie in
\[\{T(\iu\alpha,\iu\beta_i)\in\mathfrak{t}^{\mathbb{C}}\mid0\leqslant\alpha\leqslant1/2,1/2>\beta_1\geqslant\beta_2\geqslant\cdots\geqslant\beta_n\geqslant0\}.\]
Define $$|\alpha|:=\sum_{j=1}^s\alpha^j,\quad \left|\beta^j\right|:=\sum_{\{i|\beta_i^j\geqslant0\}}\beta_i^j,\quad |\beta|:=\sum_{j=1}^s\left|\beta^j\right|,\quad |\beta_i|:=\sum_{j=1}^s\beta_i^j.$$

We will prove \prettyref{thm:main2} in this section. 

\subsection{A compactness criterion}

\begin{proposition}\label{prop:compactnesscriterion}
    If a semistable parabolic $\SO_0(2,q)$-Higgs bundle $(\EE,\Phi)$ with parabolic weight $(\alpha,\beta)$ satisfies that
    \begin{itemize}
        \item[(1)] $\alpha^j>\beta_1^j$ for every $1\leqslant j\leqslant s$,

        \item[(2)] $|\alpha|-|\beta_1|<2$,

        \item[(3)]$2\operatorname{deg}(\LL)>-2+|\beta_1|-|\alpha|$,
    \end{itemize}
    then $\eta$ vanishes identically.
\end{proposition}

\begin{proof}
    Since $\alpha^j>\beta_1^j$, by \prettyref{prop:vbofSO02q} we obtain that 
    \[\eta=\sum_{l=1}^qO\left(z^{\left\lceil\alpha^j-\beta_l^j\right\rceil}\right)\cdot e^\vee\otimes f_l\otimes\dfrac{\dd z}{z}=\sum_{l=1}^qO\left(1\right)\cdot e^\vee\otimes f_l\otimes\dd z\] with respect to the chosen frame $e^\vee,f_1,\dots,f_q$ on the coordinate chart $(U,z)$ centered at $x_j\in D$. Hence $$\eta\in\mathrm{H}^0\left(X,\operatorname{Hom}(\VV,\LL^{\vee})\otimes\KK\right)\mbox{, and }\eta^{\ast}\in\mathrm{H}^0\left(X,\operatorname{Hom}(\LL,\VV)\otimes\KK\right).$$ Therefore, $\eta\circ\eta^{\ast}\in\mathrm{H}^0\left(X,\operatorname{Hom}(\LL,\LL^\vee)\otimes\KK^2\right)\cong\mathrm{H}^0\left(X,(\LL^{\vee}\otimes\KK)^2\right)$. However, $$\begin{aligned}
        &\deg((\LL^{\vee}\otimes\KK)^2)\\
        =&2\left(\deg(\LL^{\vee})+\deg(\KK)\right)\\
        =&2\left(-\deg(\LL)-2\right)\\
        <&-4+2+|\alpha|-|\beta_1|\\
        <&-2+2=0,
    \end{aligned}$$
    so $\eta\circ\eta^*=0$. Let $N$ and $I\otimes\KK$ be the subsheaves of $\VV$ and $\LL^\vee\otimes\KK$ respectively given by the kernel and the image of $\eta$, thus $\eta$ induce the following short exact sequence of sheaves
    $$0\longrightarrow N\longrightarrow\VV\longrightarrow I\otimes\KK\longrightarrow0.$$
    Let $\mathcal{N}$ denote the saturation of $N$ in $\VV$ and if $\eta\neq0$ we get that $\operatorname{rank}(\mathcal{N})=q-1$ and the saturation of $I$ in $\mathcal{L}^{\vee}$ is $\LL^{\vee}$. Let $J\otimes\KK$ be the subsheaf of $\VV\otimes\KK$ given by the image of $\eta^\ast$ and $\mathcal{J}$ is the saturation of $J$ in $\VV$. Since for any $v\in N_x$, $l\in \LL_x$, where $x\in X$, we have
    $$\begin{aligned}
        &Q_{\VV}(v,\eta^\ast(l))\\
        =&(\eta^{\vee}(l))(v)\quad(\mbox{recall that }\eta^*=q_{\VV}^{-1}\circ\eta^{\vee})\\
        =&(\eta(v))(l)\\
        =&0,
    \end{aligned}$$
    we know that $\mathcal{J}=\mathcal{N}^{\perp}$.
    Similarly, for any $l,l'\in\LL_x$,
    $$\begin{aligned}
        &Q_{\VV}(\eta^\ast(l),\eta^\ast(l'))\\
        =&(\eta^{\vee}(l))(\eta^\ast(l'))\\
        =&((\eta\circ\eta^\ast)(l'))(l)\\
        =&0,
    \end{aligned}$$
    hence $\mathcal{J}$ is an isotropic subbundle of $\VV$. Therefore, semistability tells us that $\operatorname{pardeg}(\LL)+\operatorname{pardeg}(\mathcal{J})\leqslant0$ because $\LL\oplus\mathcal{J}$ is $\Phi$-invariant. Now we have
    $$\begin{aligned}
        0=&\deg(\VV)\\
        =&\deg(N)+\deg(I\otimes\KK)\\
        \leqslant&\deg(\mathcal{N})+\deg(\LL^{\vee})-2\\
        =&\deg(\mathcal{J})-\deg(\LL)-2\quad(\mbox{by \prettyref{lemma:pardegoforthocomplement}, }\deg(\mathcal{J})=\deg(\mathcal{N}))\\
        \leqslant&\operatorname{pardeg}(\mathcal{J})+|\beta_1|-\deg(\LL)-2\\
        \leqslant&-\operatorname{pardeg}(\mathcal{L})+|\beta_1|-\deg(\LL)-2\\
        =&-2\operatorname{deg}(\mathcal{L})+|\beta_1|-|\alpha|-2\\
        <&0,
    \end{aligned}$$
    contradiction. Thus $\eta$ vanishes identically.
\end{proof}

If $(\alpha,\beta)$ satisfies the condition in \prettyref{prop:compactnesscriterion}, then $\LL\oplus0$ is $\Phi$-invariant, hence $\operatorname{pardeg}(\LL)<0$ is a necessary condition of $\mathcal{M}(\alpha,\beta,d)$ contains a stable point. Hence the interval $$J_{\alpha,\beta}:=\left(-1+\dfrac{|\beta_1|-|\alpha|}{2},-|\alpha|\right)$$
contains an integer is a necessary condition of $\mathcal{M}(\alpha,\beta,d)$ contains a stable point. 

\begin{corollary}\label{coro:cpt}
If $J_{\alpha,\beta}$ contains an integer $d$, and $\alpha^j>\beta_1^j$ for every $1\leqslant j\leqslant s$, then $\mathcal{M}(\alpha,\beta,d)$ is compact.
\end{corollary}

\begin{proof}
    Note that $J_{\alpha,\beta}$ contains an integer $d$ implies that $-1+\frac{|\beta_1|-|\alpha|}{2}<-|\alpha|$, which means that $$|\alpha|-|\beta_1|\leqslant|\alpha|+|\beta_1|<2.$$ 
    And also, $$2\operatorname{deg}(\LL)=2d>-2+|\beta_1|-|\alpha|,$$
    hence ${\alpha,\beta}$ satisfies the condition in \prettyref{prop:compactnesscriterion}, thus $\eta=0$. Therefore, the Higgs field $\Phi$ of every point in $\mathcal{M}(\alpha,\beta,d)$ is nilpotent and by the properness of Hitchin fibration \prettyref{prop:hitchinproper} this shows that $\mathcal{M}(\alpha,\beta,d)$ is compact due to Hitchin fibration is proper.
\end{proof}

\subsection{Underlying bundle of \texorpdfstring{$\mathcal{M}(\alpha,\beta,d)$}{M(α,β,d)} if \texorpdfstring{$d\in J_{\alpha,\beta}$}{d in Jα,β} and \texorpdfstring{$\alpha^j>\beta_1^j$}{αj>β1j}}

Actually, we can characterize the underlying bundle of points in $\mathcal{M}({\alpha,\beta},d)$ explicitly if $d\in J_{\alpha,\beta}$ and $\alpha^j>\beta_1^j$.

\begin{proposition}\label{prop:underbundle}
    Fix an $\SO_0(2,q)$-weight $(\alpha,\beta)$ which satisfies the condition in \prettyref{prop:compactnesscriterion}. If $(\LL^{\vee}\oplus\LL\oplus\VV,\Phi,{\alpha,\beta})$ is a semistable parabolic $\SO_0(2,q)$-Higgs bundle with $\deg(\LL)=d$ for $d\in J_{\alpha,\beta}$, then $\LL\cong\mathcal{O}(-1)$ and $\VV\cong\mathcal{O}^{\oplus q}$.
\end{proposition}

\begin{proof}
    Note that $$-1>\dfrac{|\beta_1|-|\alpha|}{2}-1>\dfrac{-2}{2}-1=-2$$
    and $-|\alpha|<0$, so $J_{\alpha,\beta}$ contains an integer if and only if $|\alpha|\in(0,1)$ and this integer must be $-1$. Hence $\LL$ must be isomorphic to $\mathcal{O}(-1)$.

    By Birkhoff--Grothendieck theorem, we can decompose $\VV$ into $\mathcal{O}(d_1)\oplus\cdots\oplus\mathcal{O}(d_q)$ for a unique $(d_1,\dots,d_q)\in\mathbb{Z}^q$ such that $d_1\geqslant d_2\geqslant\cdots\geqslant d_q$. Since $q_\VV$ induces an isomorphism between $\VV$ and its dual, we must have $d_{q+1-t}=-d_t$ for any $1\leqslant t\leqslant q$. If $d_1>0$, then by \cite[Proposition 4.2]{ramanan1981orthogonal} we can construct an isotropic subbundle $\VV_1$ of degree $d_1$ in $\VV$ with $\operatorname{rank}(\VV_1)\leqslant\operatorname{rank}(\mathcal{O}(d_1))=1$. And then $\operatorname{rank}(\VV_1)$ must be $1$ due to $d_1>0$. Therefore $\operatorname{pardeg}(\LL)+\operatorname{pardeg}(\VV_1)\leqslant0$ by semistability. This means that
    $$\begin{aligned}
        &-1+|\alpha|+d_1-|\beta_1|\leqslant\operatorname{pardeg}(\LL)+\operatorname{pardeg}(\VV_1)\leqslant0\\
        \Longrightarrow&\quad d_1\leqslant1-|\alpha|+|\beta_1|<1,
    \end{aligned}$$
    so $d_1=0=d_q$ and $\VV\cong\mathcal{O}^{\oplus q}$.
\end{proof}

\begin{remark}
    It would be possible to get other compact components with different values of $\deg(\LL)$ when relaxing the condition (W1). But for simplicity, we do not consider this here.
\end{remark}

\subsection{A linear-algebraic interpretation of stability}\label{sec:linearalgebraic}

From now on, we assume the $\SO_0(2,q)$-weight $(\alpha,\beta)$ satisfying
\begin{itemize}[leftmargin=2em, itemindent=1em]
    \item[(W1)] $\alpha^j>\beta_1^j$ for all $1\leqslant j\leqslant s$;

    \item[(W2)] $|\alpha|>|\beta|$;

    \item[(W3)] $|\alpha|+|\beta|<1$.
\end{itemize}
It is easy to verify that $(\alpha,\beta)$ satisfies the condition of \prettyref{prop:compactnesscriterion}, so by \prettyref{prop:underbundle} we get $\mathcal{M}(\alpha,\beta)$ has only one possible nonempty relative component $\mathcal{M}(\alpha,\beta,-1)$, and the underlying bundle of its point must be isomorphic to $\mathcal{O}(1)\oplus\mathcal{O}(-1)\oplus\mathcal{O}^{\oplus q}$. 

Now we examine what $$\gamma\in\mathrm{H}^0(X,\operatorname{Hom}(\mathcal{O}^{\oplus q},\mathcal{O}(-1))\otimes\KK(D))$$ really determines a semistable or stable parabolic $\mathrm{SO}_0(2,q)$-Higgs bundle.

By \prettyref{prop:stability}, we know that the parabolic $\mathrm{SO}_0(2,q)$-Higgs bundle determined by $\gamma$ is semistable if and only if
\begin{itemize}
    \item[(1)] $\operatorname{pardeg}(\mathcal{O}(-1)\oplus\VV_1)\leqslant0$ for any isotropic subbundle  $\VV_1\subset\mathcal{O}^{\oplus q}$,

    \item[(2)]
    $\operatorname{pardeg}(\mathcal{O}(1)\oplus\VV_2)\leqslant0$ for any isotropic subbundle $\VV_2\subset\mathcal{O}^{\oplus q}$ containing $\operatorname{im}{\gamma^*}$ (Note that an isotropic subbundle contains $\operatorname{im}{\gamma^*}$ implies that it is contained in $\operatorname{ker}\gamma$ so this direct sum is $\Phi$-invariant. See \prettyref{rem:isotropic} below for details.) and
    
    \item[(3)]
    $\operatorname{pardeg}(0\oplus\VV_3)\leqslant0$ for any isotropic subbundle $\VV_3 \subset\operatorname{ker}{\gamma}$.
\end{itemize}

\begin{remark}\label{rem:isotropic}
    Suppose $\VV^\prime$ is an isotropic (with respect to the bilinear form $Q$) subbundle of $\mathcal{O}^{\oplus q}$ containing $\operatorname{im}\gamma^*$. Then for any $v\in\mathcal\VV^{\prime}_x$, $l\in\mathcal{O}(1)_x$ where $x\in X$, 
    \[\begin{aligned}
        &0=Q(\gamma^*(l),v)\\
        =&(q\circ\gamma^*)(l)(v)\\
        =&\gamma^{\vee}(l)(v)\\
        =&\gamma(v)(l).
    \end{aligned}\]
    Therefore, $\gamma(v)=0$ and $\VV^{\prime}\subset\operatorname{ker}(\gamma)$.
\end{remark}

Since $\operatorname{Hom}(\mathcal{O}^{\oplus q},\mathcal{O}(-1))\otimes\KK(D)\cong\operatorname{Hom}(\mathcal{O}^{\oplus q},\mathcal{O})\otimes\mathcal{O}(s-3)$, by choosing a basis $\{e_1,\dots,e_{s-2}\}$ of $\mathrm{H}^0(X,\mathcal{O}(s-3))$, we can get a bijection

$$\begin{aligned}
    \left(\mathbb{C}^{1\times q}\right)^{s-2}&\longrightarrow \mathrm{H}^0(X,\operatorname{Hom}(\mathcal{O}^{\oplus q},\mathcal{O}(-1))\otimes\KK(D))\\
    \mathbf{A}=(A_i)_{i=1}^{s-2}&\longmapsto \sum_{i=1}^{s-2}A_i\otimes e_i=\gamma_{\mathbf{A}}.
\end{aligned}$$

Note that when $\alpha$, $\beta$ are fixed, the parabolic structure on $\mathcal{O}(1)\oplus\mathcal{O}(-1)\oplus\mathcal{O}^{\oplus q}$ is uniquely determined by $s$ isotropic flags 
\[\left(F_i^j\right)_{j=1}^{s}=\left(\left(\left(\mathcal{O}^{\oplus q}\right)_i^j\right)^{\perp}\right)_{j=1}^s\]
which correspond to the reverse isotropic flags $\left(\left(\mathcal{O}^{\oplus q}\right)_i^j\right)_{j=1}^{s}$ at $s$ marked points. Denote $\left(F_i^j\right)_{j=1}^s$ by $\mathbf{F}$. Therefore, every parabolic $\SO_0(2,q)$-Higgs bundle with the given assumptions on the weights 1-1 corresponds to an $(\mathbf{A},\mathbf{F})$.

Now fix an $(\mathbf{A},\mathbf{F})$. For any subspace $V'\subset\mathbb{C}^q$, we can define
$$\operatorname{pardeg}(V'):=\operatorname{pardeg}(V'\otimes\mathcal{O}),$$
where $V'\otimes\mathcal{O}$ is viewed as a parabolic subbundle of $\mathcal{O}^{\oplus q}$. Note that $|\operatorname{pardeg}(V')|\leqslant|\beta|$.

From this viewpoint, we can interpret the (semi-)stability condition as below. 

\begin{definition}
    For any $\mathbf{A}=(A_i)_{i=1}^{s-2}\in\left(\mathbb{C}^{1\times q}\right)^{s-2}$, we say $(\mathbf{A},\mathbf{F})$ is \textbf{semistable} if it satisfies the following two conditions:
    \begin{itemize}
        \item[(1)] there exists no isotropic subspace $V'$ of $\mathbb{C}^q$ such that $A_i^\mathrm{t}\in V'$ for all $i=1,\dots,s-2$.

        \item[(2)] if there is a coisotropic subspace $V'$ of $\mathbb{C}^q$ such that $A_i^\mathrm{t}\in V'$ for all $i=1,\dots,s-2$, then $\operatorname{pardeg}(V')\leqslant 0$.
    \end{itemize}
    
    In addition, if the inequality in (2) above is strict when $V'\neq\mathbb{C}^q$, then we say $(\mathbf{A},\mathbf{F})$ is \textbf{stable}.
\end{definition}

\begin{theorem}\label{thm:interpretion}
    For any $\SO_0(2,q)$-weight $(\alpha,\beta)$ satisfying (W1)-(W3), $(\mathbf{A},\mathbf{F})$ is semistable (resp. stable) if and only if the corresponding $\SO_0(2,q)$-Higgs bundle is semistable (resp. stable). Moreover, if $(\mathbf{A},\mathbf{F})$ is stable, then the parabolic $\SO_0(2,q)$-Higgs bundle determined by it is also stable as a parabolic $\SO(2+q,\mathbb{C})$-Higgs bundle.
\end{theorem}

\begin{proof}
    We first suppose that $(\mathbf{A},\mathbf{F})$ is semistable. It suffices to check (1)(2)(3) mentioned at the beginning of this subsection. Note that for any isotropic subbundle $\mathcal{V}_1\subset\mathcal{O}^{\oplus q}$, 
    $$\operatorname{pardeg}(\mathcal{O}(-1)\oplus\mathcal{V}_1)\leqslant|\alpha|-1+|\beta|<0\quad\mbox{by (W3)}.$$

    For any isotropic subbundle $\VV_2\subset\VV$ containing $\operatorname{im}\gamma_{\mathbf{A}}^*$, semistability of $(\mathbf{A},\mathbf{F})$ tells us that $\deg(\VV_2)\leqslant-1$. Indeed, if $\deg(\VV_2)=0$, then $\VV_2=V_2\otimes\mathcal{O}$ for some isotropic subspace $V_2\subset\mathbb{C}^q$ by Birkhoff--Grothendieck theorem and $\operatorname{im}\gamma_{\mathbf{A}}^*\subset\VV_2$ shows that $A_j^\mathrm{t}\in V_2$ for any $j$, which contradicts the semistability of $(\mathbf{A},\mathbf{F})$. Therefore, $$\operatorname{pardeg}(\mathcal{O}(1)\oplus\VV_2)\leqslant 1-|\alpha|-1+|\beta|<0\quad\mbox{by (W2)}.$$ 
    Now we fix an isotropic subbundle $\mathcal{V}_3$ of $\operatorname{ker}\gamma_{\mathbf{A}}$. If $\deg(\VV_3)\leqslant-1$, then
    $$\operatorname{pardeg}(\VV_3)\leqslant-1+|\beta|<0\quad\mbox{by (W3)}.$$
    If $\deg(\VV_3)=0$, then $\VV_3=V_3\otimes\mathcal{O}$ for some isotropic subspace $V_3\subset\mathbb{C}^q$ as above such that $A_j(V_3)=0$, therefore $A_j^{\mathrm{t}}\in(V_3)^\perp$, hence by \prettyref{lemma:pardegoforthocomplement}
    $$\operatorname{pardeg}(\VV_3)=\operatorname{pardeg}((\VV_3)^{\perp})=\operatorname{pardeg}((V_3)^{\perp})\leqslant0.$$
    
    If $(\mathbf{A},\mathbf{F})$ is not semistable, there are two possible cases.
    \begin{itemize}
        \item[(1)] There exists an isotropic subspace $V'$ of $\mathbb{C}^q$ such that $A_i^\mathrm{t}\in V'$ for all $i=1,\dots,s-2$. Then $V'\otimes\mathcal{O}$ is an isotropic subbundle containing $\operatorname{im}\gamma_{\mathbf{A}}^*$ and
        $$\operatorname{pardeg}(\mathcal{O}(1)\oplus (V'\otimes\mathcal{O}))\geqslant 1-|\alpha|-|\beta|>0\quad\mbox{by (W3)},$$
        which means the $\SO_0(2,q)$-Higgs bundle determined by $(\mathbf{A},\mathbf{F})$ is not semistable.

        \item[(2)] There exists a coisotropic subspace $V'$ of $\mathbb{C}^q$ such that $A_i^\mathrm{t}\in V'$ for all $i=1,\dots,s-2$ and $\operatorname{pardeg}(V')>0$. Then $(V')^{\perp}\otimes\mathcal{O}$ is an isotropic subbundle of $\operatorname{ker}(\gamma_{\mathbf{A}})$ and $$\operatorname{pardeg}((V')^{\perp}\otimes\mathcal{O})=\operatorname{pardeg}(V'\otimes\mathcal{O})=\operatorname{pardeg}(V')>0,$$
        which also shows the $\SO_0(2,q)$-Higgs bundle determined by $(\mathbf{A},\mathbf{F})$ is not semistable.
    \end{itemize}

    The proof of equivalence of stability is similar and we omit it. Now suppose $(\mathbf{A},\mathbf{F})$ is stable, we prove that the stable parabolic $\SO_0(2,q)$-Higgs bundle determined by $(\mathbf{A},\mathbf{F})$ is also stable as parabolic $\SO(2+q,\mathbb{C})$-Higgs bundle. Fix a $\Phi$-invariant isotropic subbundle $\mathcal{E}^\prime$.

    \begin{itemize}
        \item[(1)] If $\mathcal{E}^\prime\subset\mathcal{V}$, then $\operatorname{pardeg}(\mathcal{E}^\prime)\leqslant0$ and it is strict when $\mathcal{E}^\prime\neq0$.

        \item[(2)] If $\mathcal{E}^\prime=\mathcal{O}(-1)\oplus\VV^\prime$ for some isotropic $\VV^\prime\subset\mathcal{V}$, then $\operatorname{pardeg}(\mathcal{E}^\prime)\leqslant-1+|\alpha|+|\beta|<0$ by (W3).

        \item[(3)]  If $\mathcal{E}^\prime=\mathcal{O}(1)\oplus\VV^\prime$ for some isotropic $\VV^\prime\subset\mathcal{V}$, then $\deg(\VV^\prime)\leqslant -1$ and $\operatorname{pardeg}(\mathcal{E}^\prime)\leqslant-|\alpha|+|\beta|<0$ by (W2).
    \end{itemize}
    Therefore by \prettyref{prop:stabsonC}, $(\mathcal{O}(1)\oplus\mathcal{O}(-1)\oplus\VV,\Phi)$ corresponds to $(\mathbf{A},\mathbf{F})$ is stable as parabolic $\SO(2+q,\mathbb{C})$-Higgs bundle.
\end{proof}

\begin{corollary}\label{coro:stablepoint}
    If $s\geqslant q+2$, there exists a $\gamma\in\mathrm{H}^0(X,\operatorname{Hom}(\mathcal{O}^{\oplus q},\mathcal{O}(-1))\otimes\KK(D))$ such that the $\SO_0(2,q)$-Higgs bundle of weight $(\alpha,\beta)$ determined by it is stable and simple.
\end{corollary}

\begin{proof}
    Since $s-2\geqslant q$, we can choose $A_i$ such that $A_i^{\mathrm{t}}$ spans $\mathbb{C}^q$. Therefore for any $\mathbf{F}$, $(\mathbf{A},\mathbf{F})$ is stable. Thus it determines a stable parabolic $\SO_0(2,q)$-Higgs bundle of weight $(\alpha,\beta)$. Also $\operatorname{Aut}(\mathbb{E},\Phi)=\{\pm I_{2+q}\}\cap\SO(2+q,\mathbb{C})=Z(\SO(2+q,\mathbb{C}))\cap\operatorname{ker}\iota$ follows since $A_i$ spans $\mathbb{C}^{1\times q}$.
\end{proof}

\begin{remark}
    To show that the desired moduli space is non-empty, the strengthened assumptions (W2)(W3) are strongly used here to control the parabolic degree of the chosen isotropic subbundle. 
\end{remark}

\subsection{A GIT construction}

In this subsection, we would like to construct a space with an $\SO(2,\mathbb{C})\times\SO(q,\mathbb{C})$-linearization such that the corresponding GIT quotient is isomorphic to $\mathcal{M}(\alpha,\beta,-1)$ with fixed $\SO_0(2,q)$-weight $(\alpha,\beta)$ satisfying (W1)-(W3). It will be proven to be a projective variety and this will complete the proof of \prettyref{thm:main2}. Below we will use the concepts in Mumford's Geometric Invariant Theory \cite{mumford1994geometric}, often called GIT for short, involved in \cite[Section 3.1]{tholozan2021compact}. One can also see \cite{Newstead2013IntroductionTM} and \cite{thomas2005notes} for references. 

We fix the base field $\mathbb{C}$. Let $Y$ be a smooth quasi-projective variety with an algebraic action of a complex reductive algebraic group $G$. And let $Z$ be the kernel of this action, i.e. the subgroup of $G$ that acts trivially on $Y$. A \textbf{$G$-linearized line bundle} $\LL$ is a line bundle over $Y$ equipped with an algebraic $G$-action on the total space of $\LL$ that lifts the action on $Y$, and such that $Z$ acts trivially on $\LL$. For a polarized variety $(Y,\LL)$ with a $G$-linearization, we use $Y^{\LL}\sslash G$ to denote the GIT quotient. We also denote by $\mu_{\LL}(\lambda,x)$ the Hilbert--Mumford weight for a one parameter subgroup $\lambda\colon\mathbb{C}^*\to G$ and a point $x\in Y$.

We consider only complete isotropic flags, i.e., isotropic flags
$$0=F_0\subset F_1\subset\cdots\subset F_p=\mathbb{C}^p$$
satisfying $\dim F_i=i$ for our convenience. This corresponds to the situation of $\beta_1^j>\cdots>\beta_q^j$ for every $j=1,\dots,s$ and one can see all the discussion in this subsection can be generalized to partial flags. We denote the set of complete isotropic flags of $\mathbb{C}^p$ by $\mathcal{IF}(\mathbb{C}^p)$. When $p\geqslant2$, let $\operatorname{Gr}_i(\mathbb{C}^p)$ denote the Grassmannian of $i$-dimensional subspaces of $\mathbb{C}^p$, and define
$$\begin{aligned}
    \iota_i\colon \mathcal{IF}(\mathbb{C}^p)&\longrightarrow\operatorname{Gr}_i(\mathbb{C}^p)\\
    (F_j)_{j=0}^p&\longmapsto F_i,
\end{aligned}$$
hence $(\iota_i)_{i=1}^{p-1}$ embeds $\mathcal{IF}(\mathbb{C}^p)$ into $\prod_{i=1}^{p-1}\operatorname{Gr}_i(\mathbb{C}^p)$. Note that on every Grassmannian $\operatorname{Gr}_i(\mathbb{C}^p)$ there exists a tautological line bundle $\mathcal{O}_i(-1)$ induced from the Pl\"ucker embedding and therefore from its inverse and tensor product we get $\mathcal{O}_i(n)$ for all $n\in\mathbb{Z}$. Define 
$$\mathcal{O}(a_1,\dots,a_{p-1}):=\bigotimes_{i=1}^{p-1}\iota_i^*\mathcal{O}_i(a_i),$$
where $(a_i)_{i=1}^{p-1}\in\mathbb{Z}^{p-1}$.

The group $\SO(p,\mathbb{C})$ acts on each $\operatorname{Gr}_i(\mathbb{C}^p)$ with kernel $\pm I_p$. There is a canonical lift of this action to
the total space of $\mathcal{O}_i(1)$, such that $\pm I_p$ acts by multiplication by $(\pm 1)^{-i}$ on each fiber. Therefore, 
$\mathcal{O}_i(2n)$ is $\SO(p,\mathbb{C})$-linearized for any $n\in\mathbb{Z}$.

\begin{remark}\label{rem:differentGIT}
    Note that $\mathrm{GL}(p,\mathbb{C})$ has kernel $\mathbb{C}^*\cdot I_p$, hence in \cite[Section 3]{tholozan2021compact} they need to define the new action
     $$\begin{aligned}
         \mathrm{GL}(p,\mathbb{C})\times\mathcal{O}_i(p)&\longrightarrow\mathcal{O}_i(p)\\
         (g,v)&\longmapsto\det(g)^i\cdot g(v)=:g\cdot v
     \end{aligned}$$
     on $\mathcal{O}_i(p)$, where $g(v)$ is the natural action induced by the natural $\mathrm{GL}(p,\mathbb{C})$-action on $\mathcal{O}_i(-1)$.  This forces the kernel of $\mathrm{GL}(p,\mathbb{C})$ into acting trivially. It will involve the dimension term in computation of Hilbert--Mumford weight. But for $\SO(p,\mathbb{C})$, $\det(g)\equiv 1$ implies that $g\cdot v$ and $g(v)$ coincides and we can use $\mathcal{O}_i(2n)$ as our line bundle. This will help us  get rid of dimension terms. One may compare \prettyref{prop:flagHMweight} with \cite[Proposition 3.20]{tholozan2021compact}.
\end{remark}

Now for any one parameter subgroup $\lambda\colon\mathbb{C}^*\to \SO(p,\mathbb{C})$, it is given by $\lambda(\exp(t))=\exp(tu)$
for an endomorphism $u$ of $\mathbb{C}^p$ which can be diagonalized as $\operatorname{diag}(m_1,\dots,m_p)$ under an isotropic basis $\{v_i\}_{i=1}^p$ with $m_1,m_2,\dots,m_p\in\mathbb{Z}$ (since we need $\lambda(\exp(2\pi\iu))$ to be the identity), $m_1\geqslant m_2\geqslant\cdots\geqslant m_p$ and $m_i+m_{p+1-i}=0$. Define $$U_n(\lambda):=\bigoplus_{m_i\geqslant n}\mathbb{C}v_i.$$
Then $U_n(\lambda)$ gives an isotropic filtration of $\mathbb{C}^p$, explicitly,  $U_n(\lambda)=(U_{1-n}(\lambda))^{\perp}$.

\begin{lemma}\label{lemma:calculate}
    Suppose $\lambda\colon\mathbb{C}^*\to \SO(p,\mathbb{C})$ is a one parameter subgroup and $U_n(\lambda)$ is the isotropic filtration defined as above, then for any $F\in\operatorname{im}\iota_i$, 
    $$\mu_{\mathcal{O}_i(2m)}(\lambda,F)=\dfrac{2m}{p}\cdot\sum_{n\in\mathbb{Z}}\left[i\cdot\dim(U_n(\lambda))-p\cdot\dim(U_n(\lambda)\cap F)\right]$$
\end{lemma}

\begin{proof}
    Note that the RHS above has only finitely many nonzero terms, so it is well-defined. It follows from the example in \cite[Chapter 4,\S 6]{Newstead2013IntroductionTM} or \cite[Chapter 4,\S 4]{mumford1994geometric}. In the example mentioned above, we have
    $$\mu_{\mathcal{O}_i(1)}(\lambda,F)=-i\cdot m_p+\sum_{k=1}^{p-1}\dim(F\cap U_{m_k}(\lambda))(m_{k+1}-m_k)$$
    for $\mathrm{PSL}\left(\binom{p}{i}+1,\mathbb{C}\right)$-action induced by Pl\"ucker embedding.
    Hence under our setting, we obtain
    $$\begin{aligned}
        &\mu_{\mathcal{O}_i(2m)}(\lambda,F)\\=&2m\cdot\left[-i\cdot m_p+\sum_{k=1}^{p-1}\dim(F\cap U_{m_k}(\lambda))(m_{k+1}-m_k)\right]\\
        =&2m\cdot\left[-i\cdot m_p-\sum_{n\geqslant m_{p}+1}\dim(F\cap U_{n}(\lambda))\right]
    \end{aligned}$$
    since $U_{m_k}(\lambda)=U_{m_k-1}(\lambda)=\cdots=U_{m_{k+1}+1}(\lambda)$. Now by $\dim(F\cap U_{n}(\lambda))=i$, $\dim(U_n(\lambda))=p$ for $n\leqslant m_p$, the above formula equals to
    $$\begin{aligned}
        &2m\cdot\left[-i\cdot m_p-\sum_{n\geqslant m_{p}+1}\dim(F\cap U_{n}(\lambda))-\sum_{n\leqslant m_p}\left(\dim(F\cap U_{n}(\lambda))-\dfrac{i}{p}\cdot\dim(U_n(\lambda))\right)\right]\\
        &\qquad\qquad\qquad\qquad\qquad\qquad\qquad\\
        =&2m\cdot\left[-i\cdot m_p+\sum_{n\in\mathbb{Z}}\left(\dfrac{i}{p}\cdot\dim(U_n(\lambda))-\dim(F\cap U_{n}(\lambda))\right)-\dfrac{i}{p}\cdot\sum_{n\geqslant m_p+1}\dim(U_n(\lambda))\right]\\
        =&2m\cdot\left[-i\cdot m_p+\sum_{n\in\mathbb{Z}}\left(\dfrac{i}{p}\cdot\dim(U_n(\lambda))-\dim(F\cap U_{n}(\lambda))\right)-\dfrac{i}{p}\cdot\sum_{k=1}^{p-1}\left(k\cdot(m_k-m_{k+1})\right)\right]\\
        =&\dfrac{2m}{p}\cdot\sum_{n\in\mathbb{Z}}\left[i\cdot\dim(U_n(\lambda))-p\cdot\dim(U_n(\lambda)\cap F)\right]-2mi\left(m_p+\dfrac{1}{p}\left(\sum_{k=1}^{p-1}m_k-(p-1)m_p\right)\right)\\
        =&\dfrac{2m}{p}\cdot\sum_{n\in\mathbb{Z}}\left[i\cdot\dim(U_n(\lambda))-p\cdot\dim(U_n(\lambda)\cap F)\right].
    \end{aligned}$$
    The last ``='' above holds due to $\displaystyle\sum_{k=1}^pm_k=0$.
\end{proof}

Recall that there are only two points in $\mathcal{IF}(\mathbb{C}^2)$, we fix one point $\bullet$, and denote $\mathcal{IF}(\mathbb{C}^q)$ by $\mathcal{F}$. 

For any $\mathbf{a}=(a^j)_{j=1}^s\in\mathbb{Z}^s$, $\mathbf{b}=(b_i^j)_{1\leqslant i\leqslant q-1,1\leqslant j\leqslant s}\in\mathbb{Z}^{(q-1)\times s}$, $\mathbf{F}=(F_i^j)_{j=1}^s\in\mathcal{F}^s$, we can define 
$$\mathcal{O}(2\mathbf{a},2\mathbf{b}):=\bigotimes_{j=1}^s\left((\pi_j)^*\mathcal{O}(2 a^j)\otimes(\pi_j^\prime)^*\mathcal{O}(2b_1^j,\dots,2b_{q-1}^j)\right)$$ on \[\begin{aligned}
    \mathcal{F}^s&\longleftrightarrow(\{\bullet\}\times\mathcal{F})^s\\
    \mathbf{F}&\longleftrightarrow(\bullet,\mathbf{F})
\end{aligned}\] via the embedding from it to $\left(\operatorname{Gr}_1(\mathbb{C}^2)\times\prod_{k=1}^{q-1}\operatorname{Gr}_k(\mathbb{C}^q)\right)^s$, and $(\pi_j,\pi_j^\prime)$ denotes the $j$-th projection from $\left(\operatorname{Gr}_1(\mathbb{C}^2)\times\prod_{k=1}^{q-1}\operatorname{Gr}_k(\mathbb{C}^q)\right)^s$ to $\operatorname{Gr}_1(\mathbb{C}^2)\times\prod_{k=1}^{q-1}\operatorname{Gr}_k(\mathbb{C}^q)$. Now choose $\xi=(\xi_i^j)_{1\leqslant i\leqslant 2,1\leqslant j\leqslant s}$ and $\zeta=(\zeta_k^j)_{1\leqslant k\leqslant q,1\leqslant j\leqslant s}$ such that $\xi_2^j-\xi_1^j=a^j,\quad \zeta_{k+1}^j-\zeta_k^j=b_k^j.$ We define some notations below.
$$\|\xi\|=\sum_{j=1}^s\sum_{i=1}^2\xi_i^j,\quad\|\zeta\|=\sum_{j=1}^s\sum_{i=1}^q\zeta_i^j,$$
$$|\xi(T\cap\bullet)|=\sum_{j=1}^s\sum_{i=1}^2\xi_i^j\left(\dim(T\cap\bullet_{i-1})-\dim(T\cap\bullet_i)\right),\mbox{ for any subspace $T\subset\mathbb{C}^2$,}$$
$$|\zeta(S\cap\mathbf{F})|=\sum_{j=1}^s\sum_{i=1}^q\zeta_i^j\left(\dim(S\cap F_{i-1}^j)-\dim(S\cap F_i^j)\right),\mbox{ for any subspace $S\subset\mathbb{C}^q$}.$$

\begin{proposition}\label{prop:flagHMweight}
    For $\mathbf{F}=(F_i^j)\in\mathcal{F}^s$ and a one parameter subgroup $\lambda=(\lambda_1,\lambda_2)\colon\mathbb{C}^*\to\SO(2,\mathbb{C})\times\SO(q,\mathbb{C})$ with the associated filtration $U_n(\lambda),V_n(\lambda)$ of $\mathbb{C}^2,\mathbb{C}^q$ respectively, we have
    $$\begin{aligned}
        &\mu_{\mathbf{a},\mathbf{b}}(\lambda,\mathbf{F})\\:=&\mu_{\mathcal{O}(2\mathbf{a},2\mathbf{b})}(\lambda,\mathbf{F})\\=&\sum_{n\in\mathbb{Z}}\left(-\|\xi\|\dim(U_n(\lambda))-2|\xi(U_n(\lambda)\cap\bullet)|-\dfrac{2}{q}\|\zeta\|\dim(V_n(\lambda))-2|\zeta(V_n(\lambda)\cap\mathbf{F})|\right).
    \end{aligned}$$
\end{proposition} 

\begin{proof}
    This is a direct calculation by using \prettyref{lemma:calculate} and \cite[Proposition 3.7]{tholozan2021compact}. 
    $$\begin{aligned}
        &\mu_{\mathbf{a},\mathbf{b}}(\lambda,\mathbf{F})\\
        =&\sum_{n\in\mathbb{Z}}\sum_{j=1}^s\left(a^j\dim(U_n(\lambda))-2a^j\cdot\dim(U_n(\lambda)\cap\bullet_1)\right)\\&+\sum_{n\in\mathbb{Z}}\sum_{j=1}^s\sum_{i=1}^{q-1}\left(\dfrac{2}{q}\cdot ib_i^j\cdot \dim(V_n(\lambda))-2b_i^j\cdot\dim(V_n(\lambda)\cap F_i^j)\right)\\
        =&\sum_{n\in\mathbb{Z}}\sum_{j=1}^s\left((\xi_2^j-\xi_1^j)\dim(U_n(\lambda))-2(\xi_2^j-\xi_1^j)\cdot\dim(U_n(\lambda)\cap\bullet_1)\right)\\&+\sum_{n\in\mathbb{Z}}\sum_{j=1}^s\sum_{i=1}^{q-1}\left(\dfrac{2}{q}\cdot i(\zeta_{i+1}^j-\zeta_i^j)\cdot \dim(V_n(\lambda))-2(\zeta_{i+1}^j-\zeta_i^j)\cdot\dim(V_n(\lambda)\cap F_i^j)\right).
    \end{aligned}$$

    Since
    $$\begin{aligned}
    &\sum_{j=1}^s\sum_{i=1}^{q-1}\left(\dfrac{2}{q}\cdot i(\zeta_{i+1}^j-\zeta_i^j)\cdot \dim(V_n(\lambda))\right)\\
    =&\dfrac{2}{q}\dim(V_n(\lambda))\cdot\sum_{j=1}^s\left(\sum_{i=2}^q(i-1)\zeta_i^j-\sum_{i=1}^{q-1}i\zeta_i^j\right)\\
    =&\dfrac{2}{q}\dim(V_n(\lambda))\cdot\sum_{j=1}^s\left(q\zeta_{q}^j-\sum_{i=1}^{q}\zeta_i^j\right)\\
    =&2\dim(V_n(\lambda))\cdot\sum_{j=1}^s\zeta_{q}^j-\dfrac{2}{q}\|\zeta\|\dim(V_n(\lambda))
    \end{aligned}$$
    and
    $$\begin{aligned}
        &\sum_{j=1}^s\sum_{i=1}^{q-1}\left(-2(\zeta_{i+1}^j-\zeta_i^j)\cdot\dim(V_n(\lambda)\cap F_i^j)\right)\\
        =&-2\sum_{j=1}^s\left(\sum_{i=2}^q\zeta_i^j\cdot\dim(V_n(\lambda)\cap F_{i-1}^j)-\sum_{i=1}^{q-1}\zeta_i^j\cdot\dim(V_n(\lambda)\cap F_{i}^j)\right)\\
        =&-2|\zeta(V_n(\lambda)\cap\mathbf{F})|-2\dim(V_n(\lambda)\cap F_q^j)\cdot\sum_{j=1}^s\zeta_q^j,
    \end{aligned}$$
    we have
    $$\begin{aligned}
        &\sum_{j=1}^s\sum_{i=1}^{q-1}\left(\dfrac{2}{q}\cdot i(\zeta_{i+1}^j-\zeta_i^j)\cdot \dim(V_n(\lambda))-2(\zeta_{i+1}^j-\zeta_i^j)\cdot\dim(V_n(\lambda)\cap F_i^j)\right)\\
        =&-\dfrac{2}{q}\|\zeta\|\dim(V_n(\lambda))-2|\zeta(V_n(\lambda)\cap\mathbf{F})|\quad(\mbox{recall that }F_q^j=\mathbb{C}^q).
    \end{aligned}$$
    Similarly, we also have
    $$\begin{aligned}
        &\sum_{j=1}^s\left(a^j\dim(U_n(\lambda))-2a^j\cdot\dim(U_n(\lambda)\cap\bullet_1)\right)\\
        =&-\|\xi\|\dim(U_n(\lambda))-2|\xi(U_n(\lambda)\cap\bullet)|.
    \end{aligned}$$
    Now by taking summation along $n\in\mathbb{Z}$, we complete this proof.
\end{proof}

Suppose $\mathbb{C}^2=U\oplus U'$, where $U,U'$ are two isotropic subspaces of $\mathbb{C}^2$ and $\iota_1(\bullet)=U$. Below we identify $A\in\mathbb{C}^{1\times q}$ with an $f_A\in\operatorname{Hom}(\mathbb{C}^q,\mathbb{C}^2)$ as follows: first view $A$ as the matrix of a linear transformation under the standard basis of $U^\prime$ and $\mathbb{C}^q$ and then compose it with the embedding $U^\prime\hookrightarrow\mathbb{C}^2$. Now through the standard inner product, we obtain its dual map $f_A^{\vee}\in\operatorname{Hom}(\mathbb{C}^2,\mathbb{C}^q)$. Since $f_A$ lies in $\operatorname{Hom}(\mathbb{C}^q,U^\prime)$, $f_A^\vee$ is an element of $\operatorname{Hom}(U,\mathbb{C}^q)$ actually.

Now we consider the action
$$\begin{aligned}
\left(\SO(2,\mathbb{C})\times\SO(q,\mathbb{C})\right)\times\left(\mathbb{C}^{1\times q}\right)^{r}&\longrightarrow\left(\mathbb{C}^{1\times q}\right)^{r}\\
((g_1,g_2),(A_j)_{j=1}^{r})&\longmapsto (g_1\circ A_j\circ g_2^{-1})_{j=1}^{r}
\end{aligned}$$
With the induced trivial action of $\SO(2,\mathbb{C})\times\SO(q,\mathbb{C})$ on the trivial line bundle $\mathcal{O}$ over $\left(\mathbb{C}^{1\times q}\right)^{r}$, one can easily get
\begin{lemma}\label{lemma:baseHMweight}
    For $\mathbf{A}=(A_j)\in\left(\mathbb{C}^{1\times q}\right)^{r}$ and a one parameter subgroup $\lambda=(\lambda_1,\lambda_2)\colon\mathbb{C}^*\to\SO(2,\mathbb{C})\times\SO(q,\mathbb{C})$ with the associated filtration $U_n(\lambda),V_n(\lambda)$ of $\mathbb{C}^2,\mathbb{C}^q$ respectively, we have $\mu_{\mathcal{O}}(\lambda,\mathbf{A})=+\infty$ unless for any $1\leqslant j\leqslant r$ and $n\in\mathbb{Z}$, $f_{A_j}^\vee(U_n(\lambda))\subset V_n(\lambda)$, and in this case, $\mu_{\mathcal{O}}(\lambda,\mathbf{A})=0$.
\end{lemma}

\begin{proof}
    By definition, 
    \[\mu_{\mathcal{O}}(\lambda,\mathbf{A})=\begin{cases}
        0& \mbox{when }\lim_{t\to 0}\lambda(t)\cdot\mathbf{A}\mbox{ exists}\\
        +\infty& \mbox{when }\lim_{t\to 0}\lambda(t)\cdot\mathbf{A}\mbox{ does not exist}\\
    \end{cases}\]
    and $\lim_{t\to 0}\lambda(t)\cdot\mathbf{A}$ exists if and only if $\lim_{t\to 0}\lambda(t)\cdot A_j$ exists for all $j$. Therefore it suffices to prove for $j=1$.

    Now suppose $\lambda=(\lambda_1,\lambda_2)\colon\mathbb{C}^*\to\SO(2,\mathbb{C})\times\SO(q,\mathbb{C})$. We take $u\in U$ and $u^\prime\in U^\prime$ such that under this basis the matrix of the standard bilinear form of $\mathbb{C}^2$ is $\begin{pmatrix}
        0&1\\1&0
    \end{pmatrix}.$
    Suppose under the above basis, $u,u^\prime$, of $\mathbb{C}^2$,
    $\lambda_1(\exp(t))=\exp(t\cdot\operatorname{diag}(l,-l))$.
    Here we do not need $l>0$, i.e. this diagonalization may not define the filtration $U_n(\lambda)$. 
    
    We also let \[\lambda_2(\exp(t))=\exp(t\cdot\operatorname{diag}(m_1,m_2,\dots,m_q))\]
    be the diagonalization defined $V_n(\lambda)$, i.e. $m_1\geqslant m_2\geqslant\cdots\geqslant m_q$, $m_i+m_{q+1-i}=0$ with corresponding isotropic basis $v_1,\dots,v_q$. We fix an $A\in\mathbb{C}^{1\times q}$. Suppose under the basis $u,u^\prime,v_1,\dots,v_q$, $f_A\in\operatorname{Hom}(\mathbb{C}^q,U^\prime)$ is presented as 
    \[\begin{pmatrix}
        0&\cdots&0\\a_1&\cdots&a_q
    \end{pmatrix}\]
    
    So $\lambda(t)\cdot A$ is considered as $\lambda(t)\cdot f_A$ and
    \[\lambda(t)\cdot f_A=\begin{pmatrix}
        0&\cdots&0\\\exp(-t(l+m_1))\cdot a_1&\cdots&\exp(-t(l+m_q))\cdot a_q
    \end{pmatrix}.\]
    Thus $\lim_{t\to 0}\lambda(t)\cdot A=\lim_{t\to -\infty}\lambda(\exp(t))\cdot A=\lim_{t\to -\infty}\lambda(\exp(t))\cdot f_A$ exists if and only if for any $l>-m_i$, $a_i=0$. Or equivalently, for any $l>m_i$, $a_{q+1-i}=0$. Now under  the basis $u,u^\prime,v_1,\dots,v_q$, $f_A^{\vee}\in\operatorname{Hom}(U,\mathbb{C}^q)$ is presented as 
    \[\begin{pmatrix}
        a_q&0\\\vdots&0\\a_1&0
    \end{pmatrix}\]
    
    If $\lim_{t\to 0}\lambda(t)\cdot A$ exists, then for any $l>m_i$, $a_{q+1-i}=0$. When $n>l$, we have $f_A^\vee(U_n(\lambda))\subset f_A^\vee(U^\prime)=0\subset V_n(\lambda)$. When $n\leqslant l$, we have $m_i<l$ for any $m_i<n$. Hence $a_{q+1-i}=0$ for any $m_i<n$. Therefore \[f_A^\vee(U_n(\lambda))=f_A^\vee(U)=\mathbb{C}\cdot\sum_{i=1}^q a_{q+1-i}v_i=\mathbb{C}\cdot\sum_{m_i\geqslant n} a_{q+1-i}v_i\subset V_n(\lambda).\]

    Conversely, if $f_A^\vee(U_n(\lambda))\subset V_n(\lambda)$ for any $n\in\mathbb{Z}$. Then by taking $n=l$ we obtain that $a_{q+1-i}=0$ for any $m_i<l$, which completes the proof.
\end{proof}

\begin{remark}\label{rem:differentGITbasis}
    Similarly to \prettyref{rem:differentGIT}, the elements in $\SO(p,\mathbb{C})$ have determinant $1$ help us get rid of dimension terms. One may compare the formula of $\mu(\lambda,\mathbf{A})$ in \prettyref{lemma:baseHMweight} with that in \cite[Proposition 3.12]{tholozan2021compact}.
\end{remark}

Consider the space $E(q,r,s)=\left(\mathbb{C}^{1\times q}\right)^{r}\times\mathcal{F}^s$ with the line bundle induced from $\mathcal{O}(2\mathbf{a},2\mathbf{b})$, we still denote it by $\mathcal{O}(2\mathbf{a},2\mathbf{b})$. Denote the GIT quotient 
$$(E(q,r,s),\mathcal{O}(2\mathbf{a},2\mathbf{b}))\sslash\left(\SO(2,\mathbb{C})\times\SO(q,\mathbb{C})\right)$$ by $\mathcal{R}(q,r,s,\mathbf{a},\mathbf{b})$.

\begin{theorem}\label{thm:isomorphic}
    For an $\SO_0(2,q)$-weight $(\alpha,\beta)$ satisfying (W1)-(W3), there exists $\mathbf{a},\mathbf{b}$ such that $\mathcal{M}(\alpha,\beta,-1)$ is isomorphic to $\mathcal{R}(q,s-2,s,\mathbf{a},\mathbf{b})$.
\end{theorem}

\begin{proof}
    Since the stability condition is open, we can choose an $\SO_0(2,q)$-weight $(\alpha^\prime,\beta^\prime)$ near $(\alpha,\beta)$ such that $(\alpha^\prime)^j,(\beta^\prime)_i^j$ are all rational and $\mathcal{M}(\alpha^\prime,\beta^\prime,-1)\cong\mathcal{M}(\alpha,\beta,-1)$. Therefore, without loss of generality, we can assume that $\alpha^j,\beta_i^j$ are all rational. 

    Let $N$ be a positive integer such that $N\alpha^j,N\beta_i^j$ are all integer. Then define \[\xi=-N\alpha, \zeta=-N\beta,\] \[a^j=2N\alpha^j, b_i^j=N\left(\beta_{i}^j-\beta_{i+1}^j\right),\] \[\mathbf{a}=(a^j),\mathbf{b}=(b_i^j).\] Note that $\|\xi\|=\|\zeta\|=0$, $|\xi(U\cap \bullet)|=N|\alpha|$, $|\xi(U'\cap \bullet)|=-N|\alpha|$, $|\xi(\mathbb{C}^2\cap \bullet)|=0$, $|\zeta(V'\cap\mathbf{F})|=N\operatorname{pardeg}(V')$. Let $\lambda=(\lambda_1,\lambda_2)\colon\mathbb{C}^*\to\SO(2,\mathbb{C})\times\SO(q,\mathbb{C})$ be a one parameter subgroup with the associated filtration $U_n(\lambda),V_n(\lambda)$ of $\mathbb{C}^2,\mathbb{C}^q$ respectively. By definition, under the standard basis of $\mathbb{C}^q$, $f_A^\vee(U)=f_A^\vee(\mathbb{C}^2)$ is the subspace spanned by $A^{\mathrm{t}}$ in $\mathbb{C}^q$ for an arbitrary $A\in\mathbb{C}^{1\times q}$.

    If the $(\mathbf{A},\mathbf{F})\in E(q,s-2,s)$ is not semistable, i.e. does not correspond to a semistable parabolic $\SO_0(2,q)$-Higgs bundle in $\mathcal{M}(\alpha,\beta)$ (see the discussion at the beginning of \prettyref{sec:linearalgebraic} and also see the equivalence of semistability in \prettyref{thm:interpretion}), there are two possible cases.
    \begin{itemize}
        \item [(1)] There exists an isotropic subspace $V'$ such that $A_j^\mathrm{t}\in V'$, then consider the following filtration (note that an isotropic filtration corresponds to a unique one parameter subgroup $\lambda$):
        \[U_n(\lambda)=\begin{cases}
            \mathbb{C}^2& n\leqslant -1,\\
            U& n=0,1,\\
            0& n\geqslant2,
        \end{cases}\quad V_n(\lambda)=\begin{cases}
            \mathbb{C}^q& n\leqslant -1,\\
            (V')^{\perp}& n=0,\\
            V'& n=1,\\
            0& n\geqslant2.
        \end{cases}\]
        This filtration satisfies that $f_{A_j}^{\vee}(U_n(\lambda))\subset V_n(\lambda)$ for any $n$ and $j$. Hence, 
        \[\begin{aligned}
            &\mu_{\mathbf{a},\mathbf{b}}(\lambda,(\mathbf{A},\mathbf{F}))\\
            =&\mu_\mathcal{O}(\lambda,\mathbf{A})+\mu_{\mathcal{O}(2\mathbf{a},2\mathbf{b})}(\lambda,\mathbf{F})\quad(\mbox{by \cite[Proposition 3.7]{tholozan2021compact}})\\
            =&\mu_{\mathcal{O}(2\mathbf{a},2\mathbf{b})}(\lambda,\mathbf{F})\quad(\mbox{by \prettyref{lemma:baseHMweight}})\\
            =&-4N(|\alpha|+\operatorname{pardeg}(V'))\quad(\mbox{by \prettyref{prop:flagHMweight}})\\
            \leqslant&-4N(|\alpha|-|\beta|)<0
        \end{aligned}\]
        by (W2), which shows that $(\mathbf{A},\mathbf{F})$ is not GIT-semistable by Hilbert--Mumford criterion, c.f. \cite[Theorem 3.6]{tholozan2021compact}.

        \item[(2)] There exists a coisotropic subspace $V'\subset\mathbb{C}^q$ such that $\operatorname{pardeg}(V')>0$ and $A_j^{\mathrm{t}}\in V'$. Now construct the following isotropic filtration 
        $$U_n(\lambda)=\begin{cases}
            \mathbb{C}^2& n\leqslant 0,\\
            0& n\geqslant1,
        \end{cases}\quad V_n(\lambda)=\begin{cases}
            \mathbb{C}^q& n\leqslant -1,\\
            V'& n=0,\\
            (V')^{\perp}& n=1,\\
            0& n\geqslant2.
        \end{cases}$$
        This filtration satisfies that $f_{A_j}^{\vee}(U_n(\lambda))\subset V_n(\lambda)$ for any $n$ and $j$. Then by \prettyref{prop:flagHMweight}, \prettyref{lemma:baseHMweight} and \cite[Proposition 3.7]{tholozan2021compact} again, we obtain that
        \[\mu_{\mathbf{a},\mathbf{b}}(\lambda,(\mathbf{A},\mathbf{F}))=-4N(\operatorname{pardeg}(V'))<0,\] which shows that $(\mathbf{A},\mathbf{F})$ is not GIT-semistable by Hilbert--Mumford criterion.
    \end{itemize}

    If $(\mathbf{A},\mathbf{F})$ is not GIT-semistable, then there exists a one parameter subgroup $\lambda$ such that $\mu_{\mathbf{a},\mathbf{b}}(\lambda,(\mathbf{A},\mathbf{F}))<0$ by Hilbert--Mumford criterion. So we must have $f_{A_j}^{\vee}(U_n(\lambda))\subset V_n(\lambda)$ for any $j$ and $n$. There are two possible cases by discussing the $1$-dimensional term in $U_n(\lambda)$.
    \begin{itemize}
        \item[(1)] There exists no $n$ such that $U_n(\lambda)=U$. Then \[-\dfrac{\mu_{\mathbf{a},\mathbf{b}}(\lambda,(\mathbf{A},\mathbf{F}))}{2}=\sum_{n\in\mathbb{Z}}|\xi(U_n(\lambda)\cap\bullet)|+N\operatorname{pardeg}(V_n(\lambda))>0.\] Suppose $\#\{n|U_n(\lambda)=U'\}=m$, then $$\begin{aligned}
            0&<-m|\alpha|+\sum_{n\in\mathbb{Z}}\operatorname{pardeg}(V_n(\lambda))\\
            &\leqslant -m|\alpha|+m|\beta|+\sum_{\{n|\dim U_n(\lambda)\neq1\}}\operatorname{pardeg}(V_n(\lambda))\\
            &\leqslant\sum_{\{n|\dim U_n(\lambda)\neq1\}}\operatorname{pardeg}(V_n(\lambda))
        \end{aligned}$$by (W2), which shows that there is a coisotropic subspace $V'$, i.e. a $V_{n}(\lambda)$ for $n\leqslant0$, such that $\operatorname{pardeg}(V')>0$ and $A_j^{\mathrm{t}}\in f_{A_j}^{\vee}(\mathbb{C}^2)\subset V'$. Hence $(\mathbf{A},\mathbf{F})\in E(q,s-2,s)$ is not semistable.

        \item[(2)] There exists some $n$ such that $U_n(\lambda)=U$. By the definition of isotropic filtration, there exists $n\geqslant 1$ such that $U_n(\lambda)=U$. Then the corresponding $V_n(\lambda)$ is an isotropic subspace of $\mathbb{C}^q$, which shows that $(\mathbf{A},\mathbf{F})\in E(q,s-2,s)$ is not semistable.
    \end{itemize}

    Therefore, we get a surjective morphism $$\varphi\colon E(q,s-2,s)^{\mathrm{ss}}(\mathcal{O}(2\mathbf{a},2\mathbf{b}))\longrightarrow\mathcal{M}(\alpha,\beta,-1),$$
    where the superscript $\mathrm{ss}$ denotes semistable points. And note that two points in $E(q,s-2,s)$ map to isomorphic parabolic $\SO_0(2,q)$-Higgs bundles if and only if they are in the same $\SO(2,\mathbb{C})\times\SO(q,\mathbb{C})$-orbit, hence $\varphi$ descends to an isomorphism $\tilde{\varphi}\colon\mathcal{R}(q,s-2,s,\mathbf{a},\mathbf{b})\to\mathcal{M}(\alpha,\beta,-1)$.
\end{proof}

\subsection{Proof of \prettyref{thm:main2}}

Now \prettyref{thm:main2} is a corollary of all the discussions in \prettyref{sec:cc}. We recall its statement first.

\main*

\begin{proof}
    The compactness follows from \prettyref{coro:cpt}. The existence of stable point when $s\geqslant q+2$ follows from \prettyref{coro:stablepoint}. To show that $\mathcal{M}(\alpha,\beta)=\mathcal{M}(\alpha,\beta,-1)$ (by \prettyref{prop:underbundle}) is a projective variety over $\mathbb{C}$, by definition of the GIT quotient (see \cite[Definition 3.2]{tholozan2021compact}) that $\mathcal{R}(q,r,s,\mathbf{a},\mathbf{b})$ is a projective variety over 
$$\operatorname{Spec}\left(\mathrm{H}^0\left(E(q,r,s),\mathcal{O}\right)^{\SO(2,\mathbb{C})\times\SO(q,\mathbb{C})}\right)$$
for any $\mathbf{a},\mathbf{b}$. Hence by \cite[\href{https://stacks.math.columbia.edu/tag/0C4M}{Lemma 0C4M}]{stacks-project}, $\mathcal{R}(q,r,s,\mathbf{a},\mathbf{b})$ is a quasi-projective variety over $\operatorname{Spec}(\mathbb{C})$. By \prettyref{thm:isomorphic}, we can choose suitable $\mathbf{a},\mathbf{b}$ such that $\mathcal{M}(\alpha,\beta,-1)=\mathcal{M}(\alpha,\beta)$ is isomorphic to $\mathcal{R}(q,s-2,s,\mathbf{a},\mathbf{b})$. By \prettyref{coro:cpt}, $\mathcal{M}(\alpha,\beta,-1)$ is compact (under the complex analytic topology over its complex points), or equivalently (\cite[\uppercase\expandafter{\romannumeral12} Proposition 3.2]{SGA1}), a complete variety, i.e. a variety with proper structure map $\mathcal{M}(\alpha,\beta,-1)\to\operatorname{Spec}(\mathbb{C})$. Note that a morphism is projective if and only if it is both quasi-projective and proper (\cite[\href{https://stacks.math.columbia.edu/tag/0BCL}{Lemma 0BCL}]{stacks-project}). Therefore, $\mathcal{M}(\alpha,\beta,-1)$ is a projective variety over $\operatorname{Spec}(\mathbb{C})$, i.e. a projective variety over $\mathbb{C}$. This completes the proof of \prettyref{thm:main2}.
\end{proof}

\section{Compact components in the relative character variety}\label{sec:ccrep}

\subsection{Proof of \prettyref{thm:main} when \texorpdfstring{$s\geqslant q+2$}{s≥q+2}}

Define 
$\mathcal{W}:=\left\{(\alpha,\beta)\mbox{ is an }\SO_0(2,q)\mbox{-weight satisfying (W1)-(W3)}\right\}.$
By the non-abelian Hodge correspondence (see \prettyref{section:NAH}), \prettyref{coro:cpt} and \prettyref{coro:stablepoint} can be translated into the following theorem.

\begin{theorem}
    Assume $s\geqslant q+2$. If $(\alpha,\beta)\in\mathcal{W}$, then the relative component $$\mathfrak{X}_{h(\alpha,\beta)}^{|\alpha|-1}(\Sigma_{0,s},\SO_0(2,q))$$ is compact, non-empty, and contains an irreducible representation.
\end{theorem}

\begin{proof}
    For $s\geqslant q+2$,  there exists $(\mathcal{E},\Phi)\in\mathcal{M}(\alpha,\beta)$ which is simple, stable and stable as a parabolic $\SO(2+q,\mathbb{C})$-Higgs bundle by \prettyref{thm:interpretion} and \prettyref{coro:stablepoint}. Now by \prettyref{prop:NAH}, it corresponds to an irreducible representation through the non-abelian Hodge correspondence.
\end{proof}

Similarly to \cite[Section 5]{tholozan2021compact}, in order to get a dense representation we require the following lemma.

\begin{lemma}\label{lemma:interior}
    Assume $s\geqslant q+2$. Define $$\Omega:=\bigcup_{(\alpha,\beta)\in\mathcal{W}}\mathfrak{X}_{h(\alpha,\beta)}^{|\alpha|-1}(\Sigma_{0,s},\SO_0(2,q)).$$
    There is a full measure open subset $\mathcal{W}^\prime\subset\mathcal{W}$ such that
    $$\Omega^\prime:=\bigcup_{(\alpha,\beta)\in\mathcal{W}^\prime}\mathfrak{X}_{h(\alpha,\beta)}^{|\alpha|-1}(\Sigma_{0,s},\SO_0(2,q))\subset\Omega$$ is open in the absolute character variety $\mathfrak{X}(\Sigma_{0,s},\SO_0(2,q))$.
\end{lemma}

\begin{proof}
    Let $$\mathcal{W}^\prime=\left\{(\alpha,\beta)\in\mathcal{W}\mid\beta_1^j>\beta_2^j>\cdots>\beta_q^j,\forall 1\leqslant j\leqslant s\right\},$$
    then $\mathcal{W}^\prime$ is a full measure open subset of $\mathcal{W}$. Below we prove that $\Omega^\prime$ is an open subset of $\mathfrak{X}(\Sigma_{0,s},\SO_0(2,q))$. Take $[\rho_0]\in \mathfrak{X}_{h(\alpha,\beta)}^{|\alpha|-1}(\Sigma_{0,s},\SO_0(2,q))$ for some $(\alpha,\beta)\in\mathcal{W}'$. Note that under an isotropic basis $\mathcal{B}_j$, $\rho_0(c_j)$ can be diagonalized as $$\operatorname{diag}\left(\exp(2\pi\iu\alpha^j),\exp(-2\pi\iu\alpha^j),\exp(2\pi\iu\beta_i^j)\right),\forall1\leqslant j\leqslant s$$
    with distinct eigenvalues, hence there exists a small neighborhood $\Omega([\rho_0])$ of $[\rho_0]$ in $\mathfrak{X}(\Sigma_{0,s},\SO_0(2,q))$ such that for a fixed $[\rho]\in\Omega([\rho_0])$, $\rho(c_j)$ can be diagonalized as $$\operatorname{diag}\left(\exp(2\pi\iu(\alpha')^j),\exp(-2\pi\iu(\alpha')^j),\exp(2\pi\iu(\beta')_i^j)\right),\forall1\leqslant j\leqslant s$$
    under an isotropic basis $\mathcal{B}_j^\prime$ which has the same orientation with $\mathcal{B}_j$ for some $\SO_0(2,q)$-weight $(\alpha',\beta')\in\mathcal{W}'$ near $(\alpha,\beta)$. Note that $\operatorname{Tol}([\rho])-|\alpha'|\in\mathbb{Z}$. Since the Toledo invariant is continuous (see \prettyref{prop:tol1}), we get that $\operatorname{Tol}([\rho])$ must be $|\alpha'|-1$ for $\Omega([\rho_0])$ small enough. So $$[\rho]\in\mathfrak{X}_{h(\alpha',\beta')}^{|\alpha'|-1}(\Sigma_{0,s},\SO_0(2,q))$$ for some $(\alpha',\beta')\in\mathcal{W}'$ and this shows that $\Omega([\rho_0])\subset\Omega'$, which also means that $\Omega'$ is open in $\mathfrak{X}(\Sigma_{0,s},\SO_0(2,q))$.
\end{proof}

In \cite[Theorem 4]{Winkelmann2002GenericSO}, J. Winkelmann proved that 

\begin{proposition}\label{prop:dense}
    Let $G$ be a connected semisimple real Lie group. There exists an open neighbourhood $W$ of the identity element in $G$ and for every $k\geqslant 2$ a subset $Z_k\subset W^k$ of measure zero such that the subgroup generated by $g_1,g_2,\dots,g_k$ in $G$ is dense in $G$ for all $(g_1,g_2,\dots,g_k)\in W^k\setminus Z_k$.
\end{proposition}

Therefore, combining \prettyref{lemma:interior} and \prettyref{prop:dense}, we directly get that 

\begin{corollary}
    Assume $s\geqslant q+2$, there is an open subset $\mathcal{W}^\dprime\subset\mathcal{W}$ such that   $$\left\{(\alpha,\beta)\in\mathcal{W}^\dprime\mid\mathfrak{X}_{h(\alpha,\beta)}^{|\alpha|-1}(\Sigma_{0,s},\SO_0(2,q))\mbox{ contains a dense representation}\right\}$$
    is of full measure in $\mathcal{W}^\dprime$.
\end{corollary}

\begin{remark}
    Note that every $f\in\mathbb{R}[x,x^{-1}]$ which vanishes on $\mathbb{R}_{>0}$ must be $0$. We obtain that $\SO_0(1,1)\cong\{\operatorname{diag}(\lambda,\lambda^{-1})\mid\lambda>0\}$, which is the zero locus of a polynomial on $\SO_0(2,q)$, is only semi-algebraic but not algebraic. Therefore, $\SO_0(2,q)$ is only semi-algebraic but not a linear algebraic group over $\mathbb{R}$. Hence we cannot use \cite[Theorem 3]{Winkelmann2002GenericSO} to deduce that $\mathcal{W}^\dprime=\mathcal{W}^\prime$ and get Zariski-dense representations. Fortunately, the identity element is contained in $\Omega$ we constructed. Actually, one may also get relative components containing a dense, rather than Zariski-dense, representation when $G=\mathrm{SU}(p,q)$ by \cite[Theorem 4]{Winkelmann2002GenericSO}.
\end{remark}

\begin{remark}
    Note that if an element $g$ in $\SO_0(2,q)$ commutes with a dense subset of $\SO_0(2,q)$, then it must lie in the center of $\SO_0(2,q)$ by the continuity of the adjoint action.
    Therefore, the automorphism group of a dense representation $\rho\colon\Gamma_{0,s}\to\SO_0(2,q)$ is exactly the center of $\SO_0(2,q)$, hence the representation is irreducible.
\end{remark}

Now we would like to prove the total non-hyperbolicity of $\mathfrak{X}_{h(\alpha,\beta)}^{|\alpha|-1}(\Sigma_{0,s},\SO_0(2,q))$, the proof is a complete imitation of the original proof for $\mathrm{SU}(p,q)$, see \cite[Theorem 5.4]{tholozan2021compact}.

\begin{theorem}\label{thm:te}
    Assume $s\geqslant q+2$. If $(\alpha,\beta)\in\mathcal{W}$, then the relative component $$\mathfrak{X}_{h(\alpha,\beta)}^{|\alpha|-1}(\Sigma_{0,s},\SO_0(2,q))$$ consists of totally non-hyperbolic representations, i.e. for any $[\rho]$ in it and the homotopy class $[c]$ of an arbitrary simple closed curve $c$ on $\Sigma_{0,s}$, all eigenvalues of $\rho([c])$ have modulus $1$. 
\end{theorem}

\begin{proof}
    Denote the symmetric space $(\SO(2)\times\SO(q))\backslash\SO_0(2,q)$ of $\SO_0(2,q)$ by $\mathcal{Y}$ whose tangent space at the identity is identified with 
    $$\mathfrak{m}=\left\{\begin{pmatrix}
        0&A\\A^{\mathrm{t}}&0
    \end{pmatrix}\Bigg| A\in\mathbb{R}^{2\times q}\right\}.$$
    Following from \cite[Theorem 7.117]{knapp1996lie}, the complex structure $J$ of $\mathcal{Y}$ is given by $$\operatorname{ad}\begin{pmatrix}
        0& 1&\ \\ -1& 0&\ \\\ &\ &0_{q\times q}
    \end{pmatrix}\colon\mathfrak{m}\longrightarrow\mathfrak{m}.$$
    By taking complexification, the $\iu$-eigenspace of $J$ in $$\mathfrak{m}^{\mathbb{C}}=\left\{\begin{pmatrix}
    0&B\\-B^{\mathrm{t}}&0
\end{pmatrix}\Bigg|B\in\mathbb{C}^{2\times q}\right\}$$
    is
    $$\left\{\begin{pmatrix}
        0&0&A\\0&0&-\iu A\\-A^{\mathrm{t}}&\iu A^{\mathrm{t}}&0
    \end{pmatrix}\Bigg|A\in\mathbb{C}^{1\times q}\right\}.$$
    Therefore when using the isotropic basis, these eigenvectors are of the form 
    $$\begin{pmatrix}
        0&0&0\\0&0& A\\-A^{\mathrm{t}}&0&0
    \end{pmatrix},$$
    where $A\in\mathbb{C}^{1\times q}$. 

    Now recall the proof of the non-abelian Hodge correspondence, the map from a representation $\rho$ to an $\SO_0(2,q)$-Higgs bundle is given as follows: $\rho\colon\Gamma_{0,s}\to \SO_0(2,q)$ defines a flat bundle over $\Sigma_{0,s}$, and then one can find a harmonic metric on it which corresponds to a $\rho$-equivariant harmonic map $f\colon \widetilde{\Sigma_{0,s}}\to \mathcal{Y}$, where $\pi\colon\widetilde{\Sigma_{0,s}}\to\Sigma_{0,s}$ denotes the universal cover of $\Sigma_{0,s}$. Then the pullback of trivial principal $\SO(2,\mathbb{C})\times\SO(q,\mathbb{C})$-bundle through $f$ gives the principal bundle $\pi^*(\mathbb{E})$ over $\widetilde{\Sigma_{0,s}}$ and then descends to $\mathbb{E}$ over $\Sigma_{0,s}$. And the Higgs field is given as follows: consider the complexification of the differential of $f$, i.e. $\mathrm{d}^{\mathbb{C}}f\in\mathrm{H}^0\left(\widetilde{\Sigma_{0,s}},\operatorname{Hom}\left( T^{\mathbb{C}}\widetilde{\Sigma_{0,s}}, T^{\mathbb{C}}\mathcal{Y}\right)\right)$ and then take its $(1,0)$-part $\partial f\in\mathrm{H}^0\left(\widetilde{\Sigma_{0,s}},\operatorname{Hom}\left( T^{1,0}\widetilde{\Sigma_{0,s}}, T^{\mathbb{C}}\mathcal{Y}\right)\right)\cong\mathcal{A}^{1,0}\left(\widetilde{\Sigma_{0,s}},\pi^*(\mathbb{E})(\mathfrak{m}^{\mathbb{C}})\right)$ which can be viewed as a $\pi^*(\mathbb{E})$-valued smooth $(1,0)$-form over $\widetilde{\Sigma_{0,s}}$ . Now it descends to an element in $\mathcal{A}^{1,0}\left(\Sigma_{0,s},\mathbb{E}(\mathfrak{m}^{\mathbb{C}})\right)$ which defines the Higgs field. By the above discussion on the complex structure of $\mathcal{Y}$ we know that and \prettyref{prop:compactnesscriterion}, one find that when $(\alpha,\beta)\in\mathcal{W}$, the image of $\partial f$ is contained in $T^{1,0}\mathcal{Y}$, hence $f$ is holomorphic.

    By Harish-Chandra embedding theorem for Hermitian symmetric space, $\mathcal{Y}$ is biholomorphic to a bounded domain in $\mathbb{C}^n$, then the rest part of proof is the same as \cite[Theorem 5.4]{tholozan2021compact} by using Kobayashi distance and the contraction property of holomorphic maps, hence we omit it.
\end{proof}

Note that in the proof above, we have also proven the existence of a holomorphic $\rho$-equivariant harmonic map from $\widetilde\Sigma_{0,s}\to\mathcal{Y}$. This completes the proof of \prettyref{thm:main} in the case $s\geqslant q+2$.

\subsection{Proof of \prettyref{thm:main} for \texorpdfstring{$s\geqslant 3$}{s≥3} via 
restriction to a subsurface}

In this subsection, we deduce our main results from the results for $s\geqslant q+2$ by using the same strategy as \cite[Section 6.3]{tholozan2021compact}, that is, by restricting the representations to a subsurface. Assume $3\leqslant s<q+2$ in this subsection. Let $b$ be an oriented simple closed curve that separates $\Sigma_{0,q+2}$ into a sphere with $s$ holes $\Sigma^\prime$ and a
sphere with $q+4-s$ holes $\Sigma^\dprime$. We choose a point on $b$ as a basepoint for $\Gamma_{0,q+2}$, so as to
identify $\pi_1(\Sigma^\prime)$ and $\pi_1(\Sigma^\dprime)$ with subgroups of $\Gamma_{0,q+2}$. There is natural identifications
$$\Gamma_{0,s}\cong\pi_1(\Sigma^\prime)=\langle c_1,\dots,c_{s-1},b\mid c_1c_2\cdots c_{s-1}b=1\rangle,$$
$$\Gamma_{0,q+4-s}\cong\pi_1(\Sigma^\dprime)=\langle b,c_s,\dots,c_{q+2}\mid b^{-1}c_sc_{s+1}\cdots c_{q+2}=1\rangle$$
and an open restriction map
$$\begin{aligned}
    \operatorname{Res}\colon \mathfrak{X}(\Sigma_{0,q+2},\SO_0(2,q))&\longrightarrow\mathfrak{X}(\Sigma_{0,s},\SO_0(2,q))\\ [\rho]&\longmapsto\left[\rho|_{\pi_1(\Sigma^\prime)}\right].
\end{aligned}$$

Now let $\Omega^\prime$ be the open subset in $\mathfrak{X}(\Sigma_{0,q+2},\SO_0(2,q))$ we constructed in \prettyref{lemma:interior}, and then define $\Omega^\dprime\subset\Omega^\prime$ to be the non-empty open subset in $\Omega^\prime$ such that $\rho(b)$ is diagonalizable with distinct eigenvalues.

\begin{theorem}
    For every class of representation $[\rho]$ in the domain 
    $$\operatorname{Res}(\Omega^\dprime)\subset\mathfrak{X}(\Sigma_{0,s},\SO_0(2,q)),$$ the connected component of $[\rho]$ in its relative character variety is compact and contained in $\operatorname{Res}(\Omega^\dprime)$.
\end{theorem}

\begin{proof}
    Let $[\rho_0]$ be a class of representation in $\Omega^\dprime$. Denote respectively by $\rho_0^\prime$ and $\rho_0^\dprime$ the restrictions of $\rho_0$ to $\pi_1(\Sigma^\prime)$ and $\pi_1(\Sigma^\dprime)$. Define 
    $$h=(\rho_0(c_1),\dots,\rho_0(c_{q+2})),$$
    $$h^\prime=(\rho_0(c_1),\dots,\rho_0(c_{s-1}),\rho_0(b)),h^\dprime=(\rho_0(b^{-1}),\rho_0(c_s),\dots,\rho_0(c_{q+2})).$$
    Let $K$ denote the subset of $\mathfrak{X}_h(\Sigma_{0,q+2},\SO_0(2,q))$ consisting of representations $[\rho]$ such that $\rho(b)$ is conjugate to $\rho_0(b)$. Since $\rho_0(b)$ is diagonalizable, its conjugation orbit is closed. Therefore, $K$ is the preimage of a closed set through a continuous map, hence it is closed in $\mathfrak{X}_h(\Sigma_{0,q+2},\SO_0(2,q))$, which implies that $K$ is compact. Moreover, $K\subset\Omega^\dprime$.

    By definition, the restriction map $\operatorname{Res}$ sends $K$ to $\mathfrak{X}_{h^\prime}(\Sigma_{0,s},\SO_0(2,q))$. It suffices to prove that it is surjective. 

    For any $[\rho_1^\prime]\in\mathfrak{X}_{h^\prime}(\Sigma_{0,s},\SO_0(2,q))$, we know that $\rho_1^\prime(b)\in C(\rho_0(b))$. Let $g\in\SO_0(2,q)$ such that $\rho_1^{\prime}(b)=g\rho_0(b)g^{-1}$. Then we can define $\rho_1\colon\Gamma_{0,q+2}\to\SO_0(2,q)$ such that
    $$\rho_1|_{\pi_1(\Sigma^\prime)}=\rho_1^\prime,\quad\rho_1|_{\pi_1(\Sigma^\dprime)}=g\rho_0^\dprime g^{-1},$$
    and $[\rho_1]\in K$. Therefore, $\operatorname{Res}|_K\colon K\to \mathfrak{X}_{h^\prime}(\Sigma_{0,s},\SO_0(2,q))$ is surjective.
\end{proof}

Now apply \prettyref{prop:dense} and \prettyref{thm:te} again, we show that the representations in $\mathfrak{X}_{h^\prime}(\Sigma_{0,s},\SO_0(2,q))$ have properties (1) and (2) in \prettyref{thm:main}. Finally, the proof of property (3) for the representations in $\mathfrak{X}_{h^\prime}(\Sigma_{0,s},\SO_0(2,q))$ follows from \cite[Proposition 6.7]{tholozan2021compact} directly.


\nocite{*}
\bibliographystyle{plain}
\bibliography{bib}

\end{document}